\documentclass[12pt]{amsart}

\usepackage{amsmath, amssymb, amscd, newlfont}
\newcommand{\bbH}{{\mathbb{H}}}
\newcommand{\bbQ}{{\mathbb{Q}}}
\newcommand{\bbC}{{\mathbb{C}}}

\newcommand{\bbR}{{\mathbb{R}}}
\newcommand{\bbZ}{{\mathbb{Z}}}

\newcommand{\fraka}{{\mathfrak{a}}}

\numberwithin{equation}{section}

\newtheorem{Prop}[equation]{Proposition}
\newtheorem{Lem}[equation]{Lemma}
\newtheorem{Def}[equation]{Definition}
\newtheorem{Thm}[equation] {Theorem}
\newtheorem{Cor}[equation]{Corollary}

\title
[
On $m$--fold Holomorphic Differentials and Modular Forms
]
{On $m$--fold Holomorphic Differentials and Modular Forms} 

\author{Damir Miko\v c and Goran Mui\'c }
\address{Department of Teacher Education Studies in Gospi\' c,
  University of Zadar, dr. Ante Star\v cevi\' ca 12,
53000 Gospi\' c, Croatia}
\email{damir.mikoc@gmail.com}
\address{
Department of Mathematics, Faculty of Science, 
University of Zagreb,
Bijeni\v cka 30, 10000 Zagreb,
Croatia}
 \email{gmuic@math.hr}

\begin{document}

\begin{abstract}
  Let $\Gamma$ be the Fuchsian group of the first kind. For an even integer $m\ge 4$,  we study $m/2$--holomorphic differentials in terms of space of
  (holomorphic) cuspidal modular forms $S_m(\Gamma)$. We also give in depth study of Wronskians of cuspidal modular forms and their divisors.
  \end{abstract}
\thanks{The  author acknowledges Croatian Science Foundation grant IP-2018-01-3628.}
\maketitle

\section{Introduction}

Let $\Gamma$ be the Fuchsian group of the first kind  \cite[Section 1.7, page 28]{Miyake}.
 Examples of such groups are the important modular groups such as  $SL_2(\Bbb Z)$ and its congruence subgroups
 $\Gamma_0(N)$,  $\Gamma_1(N)$, and $\Gamma(N)$  \cite[Section 4.2]{Miyake}. Let $\Bbb H$ be the complex upper half-plane. 
The quotient $\Gamma\backslash \Bbb H$ can be compactified by adding a finite number of $\Gamma$-orbits of points in $\mathbb R\cup \{\infty\}$ called cusps of $\Gamma$ and we obtain a compact
Riemann surface which will be denoted by $\mathfrak{R}_\Gamma$. For $l\ge 1$, let $H^l\left(\mathfrak R_\Gamma\right)$ be the space of all holomorphic differentials on $\mathfrak{R}_\Gamma$ (see \cite{FK}, or
Section \ref{mhd}) in this paper).

Let $m\ge 2$ be an even integer.  Let
 $S_m(\Gamma)$ be the space of (holomorphic) cusp forms of weight $m$ (see Section \ref{prelim}). It is well--known that
$S_2(\Gamma)$ is naturally isomorphic to the vector space $H^1\left(\mathfrak R_\Gamma\right)$ (see \cite[Theorem 2.3.2]{Miyake}). This is employed on many instances in studying various properties of
modular curves (see for example \cite[Chapter 6]{ono}).  In this paper we study the generalization of this concept to the holomorphic differentials of higher order.

For an even integer $m\ge 4$, in general, the space $S_m(\Gamma)$ is too big to be isomorphic to $H^{m/2}\left(\mathfrak R_\Gamma\right)$ due to presence of cusps and elliptic
points. So, in general we define a
subspace 

$$
  S^H_m(\Gamma)=\left\{f\in S_m(\Gamma); \ \  \text{$f=0$ or $f$ satisfies (\ref{int-mhd-8})} \right\},
$$
 where
\begin{equation}\label{int-mhd-8}
\mathfrak c_f\ge \sum_{\mathfrak a\in \mathfrak
R_\Gamma, \ elliptic} \left[\frac{m}{2}(1-1/e_{\mathfrak  a})\right]\mathfrak a+ \left(\frac{m}{2} -1\right)\sum_{\mathfrak b\in \mathfrak
  R_\Gamma, \ cusp}  \mathfrak b.
\end{equation}
The integral divisor $\mathfrak c_f$ is defined in Lemma \ref{prelim-1} while the multiplicities $e_{\mathfrak  a}$ are defined in Section \ref{prelim}.
Now, we have the following result (see Section \ref{mhd-cont}):

\begin{Thm} The usual map $f\longmapsto \omega_f$  from the space of all cuspidal modular form into space of meromorphic differentials (see \cite[Theorem 2.3.3]{Miyake}) induces the isomorphism
  of  $S^H_m(\Gamma)$ onto $H^{m/2}\left(\mathfrak R_\Gamma\right)$. 
  \end{Thm}

We study the space  $S^H_m(\Gamma)$ in detail in Section (see Section \ref{mhd-cont}).  The main results are contained in a very detailed Lemma \ref{mhd-9} and Theorem \ref{mhd-10}. 
We recall (see \cite[III.5.9]{FK} or Definition \ref{mhd-def} that  $\mathfrak a \in \mathfrak R_\Gamma$ is a $m/2$-Weierstrass point if
there exists a non--zero $\omega\in H^{m/2}\left(\mathfrak R_\Gamma\right)$ such that
$$
\nu_{\mathfrak a}(\omega)\ge \dim H^{m/2}\left(\mathfrak R_\Gamma\right).
$$

Equivalently \cite[Proposition III.5.10]{FK} , if 
$$
\nu_{\mathfrak a}\left(W\left(\omega_1, \ldots, \omega_t\right)\right)\ge 1,
$$
where $W\left(\omega_1, \ldots, \omega_t\right)$ is the Wronskian of holomorphic differential forms $\omega_1, \ldots, \omega_t$ (see Section \ref{mhd}). 

When $m=2$ we speak about classical Weierstrass points. So, $1$-Weierstrass points are simply Weierstrass points. Weierstrass points on modular curves are very-well studied
(see for example \cite[Chapter 6]{ono}, \cite{neeman}, \cite{Ogg}, \cite{pete-1}, \cite{pete-2}, \cite{roh}). Higher--order Weierstrass points has not been
not studied much (see for example \cite{neeman}).

\vskip .2in
The case $m\ge 4$ is more complex. We recall that $\mathfrak R_\Gamma$ is hyperelliptic if $g(\Gamma)\ge 2$, and there is a degree two map onto $\mathbb P^1$.
By general theory \cite[Chapter VII, Proposition 1.10]{Miranda}, if $g(\Gamma)=2$, then $\mathfrak R_\Gamma$ is hyperelliptic.
If $\mathfrak R_\Gamma$ is not hyperelliptic, then $\dim S_2(\Gamma)= g(\Gamma)\ge 3$, and  the regular  map
$\mathfrak R_\Gamma\longrightarrow \mathbb P^{g(\Gamma)-1}$ attached to a canonical divisor $K$  is an isomorphism  onto its image
\cite[Chapter VII, Proposition 2.1]{Miranda}.

\vskip .2in

Let $\Gamma=\Gamma_0(N)$, $N\ge 1$. Put $X_0(N)=\mathfrak R_{\Gamma_0(N)}$. We recall that  $g(\Gamma_0(N))\ge 2$ unless
$$
\begin{cases}
N\in\{1-10, 12, 13, 16, 18, 25\} \ \ \text{when $g(\Gamma_0(N))=0$, and}\\
N\in\{11, 14, 15, 17, 19-21, 24, 27, 32, 36, 49\} \ \ \text{when
	$g(\Gamma_0(N))=1$.}
\end{cases}
$$
Let $g(\Gamma_0(N))\ge 2$. Then,  we remark that Ogg \cite{Ogg} has determined all $X_0(N)$ which are hyperelliptic curves.
In view of Ogg's paper, we see that $X_0(N)$ is {\bf not hyperelliptic} for
$N\in \{34, 38, 42, 43, 44, 45, 51-58, 60-70\}$ or $N\ge 72$. This implies $g(\Gamma_0(N))\ge 3$.

\vskip .2in
We prove the following result (see Theorem \ref{cts-50000}) 

\vskip .2in
 \begin{Thm}\label{cts-50000-int}  Let $m\ge 4$ be an even integer. Assume that  $\mathfrak R_\Gamma$ is not hyperelliptic. Then, we have
   $$
   S_{m, 2}^H(\Gamma)= S_m^H(\Gamma),
   $$
   where we denote the  subspace  $S_{m, 2}^H(\Gamma)$  of $S_m^H(\Gamma)$  spanned by all monomials
    $$
   f_0^{\alpha_0}f_1^{\alpha_1}\cdots f_{g-1}^{\alpha_{g-1}}, \ \  \alpha_i\in \mathbb Z_{\ge 0}, \ \sum_{i=0}^{g-1} \alpha_i=\frac{m}{2}.
   $$
   Here $f_0, \ldots, f_{g-1}$, $g=g(\Gamma)$, is a basis of  $S_2(\Gamma)$
 \end{Thm}

 \vskip .2in
 The criterion is given by the following corollary (see Corollary \ref{cts-500000}):

 \vskip .2in 
 \begin{Cor}\label{cts-50000-int}  Let $m\ge 4$ be an even integer. Assume that  $\mathfrak R_\Gamma$ is not hyperelliptic. Assume that $\mathfrak a_\infty$ is a cusp for $\Gamma$.
   Let us select a basis $f_0, \ldots, f_{g-1}$, $g=g(\Gamma)$, of $S_2(\Gamma)$. Compute $q$--expansions of all monomials
   $$
   f_0^{\alpha_0}f_1^{\alpha_1}\cdots f_{g-1}^{\alpha_{g-1}}, \ \  \alpha_i\in \mathbb Z_{\ge 0}, \ \sum_{i=0}^{g-1} \alpha_i=\frac{m}{2}.
   $$
   Then,  $\mathfrak a_\infty$ is {\bf not} a   $\frac{m}{2}$--Weierstrass point if and only if there exist a basis of the space of all such monomials,
   $F_1, \ldots F_t$, $t=\dim{S_m^H(\Gamma)}=(m-1)(g-1)$
   (see Lemma \ref{mhd-9} (v)), such that their $q$--expansions are of the form
      $$
      F_u=a_uq^{u+m/2-1}+ \text{higher order terms in $q$}, \ \ 1\le u\le t,
      $$
      where 
      $$
      a_u\in \mathbb C, \ \  a_u\neq 0.
      $$
\end{Cor}
This is useful for explicit computations in SAGE at least when $\Gamma=\Gamma_0(N)$.  We give examples in Section \ref{cts} (see
 Propositions \ref{cts-5001} and \ref{cts-5002}).   A different more theoretical criterion is contained in Theorem \ref{mhd-10}.
 
\vskip .2in 
Various other aspects of modular curves has
been studied in \cite{BKS}, \cite{bnmjk}, \cite{sgal}, \cite{ishida}, \cite{Muic}, \cite{MuicMi}, \cite{mshi} and \cite{yy}. 
We continue the approach presented in \cite{Muic1}, \cite{Muic2}, and \cite{MuicKodrnja}. In the proof of Theorem \ref{cts-50000}  we give an explicit construction of a
higher order canonical map
i.e., a map attached to divisor $\frac{m}{2}K$, where $K$ is a canonical divisor of $\mathfrak R_\Gamma$. The case $m=2$ is studied in depth in many papers (see for example \cite{sgal}). 

\vskip .2in 

In Section \ref{wron}  we deal with a generalization of the usual notion of the Wronskian  of cuspidal modular forms \cite{roh}, (\cite{ono}, 6.3.1),  (\cite{Muic}, the proof of
Theorem 4-5),
and (\cite{Muic2},
Lemma 4-1).  The main result of the section is Proposition  \ref{wron-2} which in the most important case has the following form:

\begin{Prop} Let $m\ge 1$. Then, for any sequence 
$f_1, \ldots, f_k\in M_m(\Gamma)$,
the Wronskian
$$
W\left(f_1, \ldots, f_k\right)(z)\overset{def}{=}\left|\begin{matrix}
f_1(z) &
\cdots & f_{k}(z) \\
\frac{df_1(z)}{dz} &
\cdots & \frac{df_{k}(z)}{dz} \\
&\cdots & \\
\frac{d^{k-1}f_1(z)}{dz^{k-1}} &
\cdots & \frac{d^{k-1}f_{k}(z)}{dz^{k-1}} \\
\end{matrix}\right|
$$
is a cuspidal modular form in $S_{k(m+k-1)}(\Gamma)$ if $k\ge 2$. If $f_1, \ldots, f_k$ are linearly independent,
then  $W\left(f_1, \ldots, f_k\right)\neq 0$.   
\end{Prop}

\vskip .2in
What is new and deep is the computation of the divisor of $W\left(f_1, \ldots, f_k\right)$ (see Section \ref{wron-cont}). The main results are Proposition \ref{wron-cont-2}
and Theorem  \ref{wron-6}. A substantial example has been given in Section \ref{lev0} in the case of $\Gamma=SL_2(\mathbb Z)$ (see Proposition \ref{lev0-4}).

\vskip .2in

We would  like to thank I. Kodrnja for her help with the SAGE system. Also we would like to thank M. Kazalicki and F. Najman for some useful discussions about modular
forms and curves in general.

\section{Preliminaries}\label{prelim}

In this section we recall necessary facts about modular forms and their divisors \cite{Miyake}.
We follow the exposition in (\cite{Muic2}, Section 2).

\vskip .2in

 Let $\mathbb H$ be the upper half--plane.
Then the group $SL_2(\bbR)$  acts on $\mathbb H$  as follows:
$$ g.z=\frac{az+b}{cz+d}, \ \ g=\left(\begin{matrix}a & b\\ c & d
\end{matrix}\right)\in SL_2(\bbR).
$$
We let $j(g, z)=cz+d$. The function $j$ satisfies the cocycle identity:
\begin{equation}\label{cocycle}
j(gg', z)=j(g, g'.z) j(g', z).
\end{equation}

Next, $SL_2(\bbR)$--invariant measure on $\mathbb H$ is defined by $dx dy
/y^2$, where the coordinates on $\mathbb H$ are written in a usual way 
$z=x+\sqrt{-1}y$, $y>0$. A discrete subgroup $\Gamma\subset
SL_2(\bbR)$ is called a Fuchsian group of the first kind if 
$$
\iint _{\Gamma \backslash \mathbb H} \frac{dxdy}{y^2}< \infty.
$$
Then, adding a finite number of points 
in $\bbR\cup \{\infty\}$ called cusps, $\cal F_\Gamma$ can be
compactified. In this way we obtain a compact Riemann surface 
$\mathfrak R_\Gamma$. One of the most important examples are the groups
$$
\Gamma_0(N)=\left\{\left(\begin{matrix}a & b \\ c & d\end{matrix}\right) \in SL_2(\mathbb Z); \ \
    c\equiv 0 \ (\mathrm{mod} \ N)\right\}, \ \ N\ge 1.
$$
We write $X_0(N)$ for    $\mathfrak R_{\Gamma_0(N)}$.

Let $\Gamma$ be a Fuchsian group of the first kind. 
 We consider the space $M_m(\Gamma)$
(resp., $S_m(\Gamma)$)  of all 
modular (resp., cuspidal) forms of weight $m$;  this is the space
of all holomorphic functions $f: \mathbb H\rightarrow \bbC$ 
such that $f(\gamma.z)=j(\gamma, z)^m f(z)$ 
($z\in \mathbb H$, $\gamma\in \Gamma$) which are holomorphic (resp., holomorphic and vanish) at
every cusp for $\Gamma$. We also need the following obvious property: for 
$f, g\in M_m(\Gamma, \chi)$, $g\neq 0$, the quotient $f/g$ is a meromorphic function on  $\mathfrak R_{\Gamma}$.

\vskip .2in 

Next, we recall from (\cite{Miyake}, 2.3) some notions related to the theory of divisors 
of modular forms of even weight $m\ge 2$ and state a preliminary result.

Let $m\ge 2$ be an even integer and $f\in M_{m}(\Gamma)-\{0\}$. Then,
$\nu_{z-\xi}(f)$ denotes the order of the holomorphic function $f$ at $\xi$.
For each $\gamma\in \Gamma$, the functional equation
$f(\gamma.z)=j(\gamma, z)^m f(z)$, $z\in \mathbb H$, shows that 
$\nu_{z-\xi}(f)=\nu_{z-\xi'}(f)$ where $\xi'=\gamma.\xi$. 
Also, if we let
$$
e_{\xi} =\# \left(\Gamma_{\xi}/\Gamma\cap \{\pm 1\}\right),
$$
then $e_{\xi}=e_{\xi'}$. The point $\xi\in \mathbb H$ is elliptic if $e_\xi>1$. Next, following (\cite{Miyake}, 2.3), we define
$$
\nu_\xi(f)=\nu_{z-\xi}(f)/e_{\xi}.
$$
Clearly, $\nu_{\xi}=\nu_{\xi'}$, and we may let 
$$
\nu_{\mathfrak a_\xi}(f)=\nu_\xi(f),
$$ 
where 
$$\text{$\mathfrak a_\xi\in \mathfrak R_\Gamma$ is the projection of $\xi$
to $\mathfrak R_\Gamma$,} 
$$
a notation we use throughout this paper.

If $x\in \bbR\cup \{\infty\}$ is a cusp for $\Gamma$, then we define 
$\nu_x(f)$ as follows. Let $\sigma\in SL_2(\bbR)$ such that
$\sigma.x=\infty$. We write 
$$
\{\pm 1\} \sigma \Gamma_{x}\sigma^{-1}= \{\pm 1\}\left\{\left(\begin{matrix}1 & lh'\\ 0 &
    1\end{matrix}\right); \ \ l\in \bbZ\right\},
$$
where $h'>0$. Then we write the Fourier expansion of $f$ at $x$ as follows:
$$
(f|_m \sigma^{-1})(\sigma.z)= \sum_{n=1}^\infty a_n e^{2\pi
  \sqrt{-1}n\sigma.z/h'}.
$$

We let 
$$
\nu_x(f)=l\ge 0,
$$
where $l$ is defined by $a_0=a_1=\cdots =a_{l-1}=0$, $a_l\neq 0$. One
easily see that this definition does not depend on $\sigma$. Also, 
if $x'=\gamma.x$, then
$\nu_{x'}(f)=\nu_{x}(f)$. Hence, if $\mathfrak b_x\in \mathfrak
R_\Gamma$ is a cusp corresponding to $x$, then we may define
$$
\nu_{\mathfrak b_x}=\nu_x(f). 
$$

Put
$$
\mathrm{div}{(f)}=\sum_{\mathfrak a\in \mathfrak
R_\Gamma} \nu_{\mathfrak a}(f) \mathfrak a \in \ \  \bbQ\otimes \mathrm{Div}(\mathfrak R_\Gamma),
$$
where $\mathrm{Div}(\mathfrak R_\Gamma)$ is the group of (integral) divisors on 
$\mathfrak R_\Gamma$.

Using (\cite{Miyake}, 2.3), this sum is finite i.e., $ \nu_{\mathfrak a}(f)\neq 0$
for only a finitely many points. We let 
$$
\mathrm{deg}(\mathrm{div}{(f)})=\sum_{\mathfrak a\in \mathfrak
R_\Gamma} \nu_{\mathfrak a}(f).
$$

Let $\mathfrak d_i\in \bbQ\otimes \mathrm{Div}(\mathfrak R_\Gamma)$, $i=1, 2$. Then we say 
that $\mathfrak d_1\ge \mathfrak d_2$ if their difference $\mathfrak d_1 - \mathfrak d_2$ 
belongs to $\mathrm{Div}(\mathfrak R_\Gamma)$ and is non--negative in the usual sense.

\begin{Lem}\label{prelim-1} Assume that $m\ge 2$ is an even integer. Assume that 
$f\in  M_m(\Gamma)$, $f\neq 0$. Let $t$ be the number of inequivalent cusps  for $\Gamma$.  Then we have the following:
\begin{itemize}

\item[(i)] For $\mathfrak a\in \mathfrak
R_\Gamma$, we have  $\nu_{\mathfrak a}(f) \ge 0$.

\item [(ii)]  For a cusp $\mathfrak a\in \mathfrak R_\Gamma$, we have that 
$\nu_{\mathfrak a}(f)\ge 0$ is an integer.

\item[(iii)] If  $\mathfrak a\in \mathfrak
R_\Gamma$ is not an elliptic point or a cusp, then $\nu_{\mathfrak a}(f)\ge 0$
is an integer.  If  $\mathfrak a\in \mathfrak
R_\Gamma$ is an elliptic point, then $\nu_{\mathfrak a}(f)-\frac{m}{2}(1-1/e_{\mathfrak a})$ is 
an integer.

\item[(iv)]Let $g(\Gamma)$ be the genus of $ \mathfrak R_\Gamma$. Then 
 \begin{align*}
\mathrm{deg}(\mathrm{div}{(f)})&= m(g(\Gamma)-1)+ \frac{m}{2}\left(t+ \sum_{\mathfrak a\in \mathfrak
  R_\Gamma, \ \ elliptic} (1-1/e_{\mathfrak a})\right)\\
&= \frac{m}{4\pi} \iint_{\Gamma \backslash \mathbb H} \frac{dxdy}{y^2}.
\end{align*}

\item[(v)] Let $[x]$ denote the largest integer $\le x$ for $x\in
  \bbR$.  Then

\begin{align*}
\dim S_m(\Gamma) &=
\begin{cases} (m-1)(g(\Gamma)-1)+(\frac{m}{2}-1)t+ \sum_{\substack{\mathfrak a\in \mathfrak
      R_\Gamma, \\ elliptic}} \left[\frac{m}{2}(1-1/e_{\mathfrak a})\right], \ \ \text{if $m\ge 4$,}\\
  g(\Gamma), \ \ \text{if $m=2$.}\\
  \end{cases}\\
\dim M_m(\Gamma)&=\begin{cases} \dim S_m(\Gamma)+t, \ \ \text{if $m\ge 4$, or $m=2$ and $t=0$,}\\
\dim S_m(\Gamma)+t-1=g(\Gamma)+t-1,\ \ \text{if $m=2$ and $t\ge 1$.}\\
\end{cases} 
\end{align*}

\item[(vi)] There exists an integral divisor $\mathfrak c'_f\ge 0$ of degree
  $$
  \begin{cases}
    \dim M_m(\Gamma)+ g(\Gamma)-1, \ \ \text{if $m\ge 4$, or $m=2$ and $t\ge 1$,}\\
    2(g(\Gamma)-1), \ \ \text{if $m=2$ and $t=0$}
  \end{cases}
  $$
  such that
\begin{align*}
\mathrm{div}{(f)}= & \mathfrak c'_f+ \sum_{\mathfrak a\in \mathfrak
R_\Gamma, \ \ elliptic} \left(\frac{m}{2}(1-1/e_{\mathfrak a}) -
\left[\frac{m}{2}(1-1/e_{\mathfrak
    a})\right]\right)\mathfrak a.
\end{align*}
\item[(vii)] Assume that $f\in S_m(\Gamma)$. Then, the  integral divisor defined by
 $ \mathfrak c_f\overset{def}{=}\mathfrak c'_f-
\sum_{\substack{\mathfrak b \in \mathfrak
R_\Gamma, \\ cusp}} \mathfrak b$ satisfies  $\mathfrak c_f\ge 0$  and its degree is given by
$$
  \begin{cases}
    \dim S_m(\Gamma)+ g(\Gamma)-1; \ \ \text{if $m\ge 4$,}\\
    2(g(\Gamma)-1); \ \ \text{if $m=2$.}
  \end{cases}
  $$
\end{itemize}
\end{Lem}
\begin{proof} The claims (i)--(v) are standard (\cite{Miyake}, 2.3, 2.4, 2.5). The claim (vi) follows from (iii), 
(iv), and (v) (see Lemma 4-1 in \cite{Muic}). Finally, (vii) follows from (vi).
\end{proof}

\section{Holomorhic Differentials and $m$--Weierstrass Points on $\mathfrak R_\Gamma$}\label{mhd}

Let $\Gamma$ be a Fuchsian group of the first kind. We let $D^m\left(\mathfrak R_\Gamma\right)$ (resp., $H^m\left(\mathfrak R_\Gamma\right)$)be the space of meromorphic (resp., holomorphic)
differential of degree $m$ on $\mathfrak R_\Gamma$ for each $m\in \mathbb Z$. We recall that $D^0\left(\mathfrak R_\Gamma\right)=\mathbb C\left(\mathfrak R_\Gamma\right)$, and
$D^m\left(\mathfrak R_\Gamma\right)\neq 0$ for all other  $m\in \mathbb Z$. In fact, if we fix a non--zero $\omega \in  D^1\left(\mathfrak R_\Gamma\right)$, then
$D^m\left(\mathfrak R_\Gamma\right)=\mathbb C\left(\mathfrak R_\Gamma\right)\omega^n$. We have  the following:
\begin{equation}\label{mhd-1}
  \deg{\left(\mathrm{div}{(\omega)}\right)}= 2m(g(\Gamma)-1), \ \ \omega \in D^m\left(\mathfrak R_\Gamma\right),  \ \omega\neq 0.
\end{equation}

We shall be interested in the case $m\ge 1$, and in holomorphic differentials. We recall \cite[Proposition III.5.2]{FK} that 
\begin{equation}\label{mhd-2}
  \dim H^m\left(\mathfrak R_\Gamma\right)=
  \begin{cases}  0 \ \ &\text{if} \ \ m\ge 1, g(\Gamma)=0;\\
     g(\Gamma)    \ \ &\text{if} \ \ m=1, \ g(\Gamma)\ge 1;\\
   g(\Gamma)    \ \ & \text{if} \ \ m\ge 2, \ g(\Gamma)= 1;\\
     (2m-1)\left(g\left(\mathfrak R_\Gamma\right)-1\right)    \ \ &\text{if} \ \ m\ge 2, \ g(\Gamma)\ge 2.\\           
  \end{cases}
  \end{equation}

This follows easily from Riemann-Roch theorem. Recall that a canonical class $K$ is simply a divisior on any non--zero meromorphic form $\omega$ on $\mathfrak R_\Gamma$. Different choices of a
$\omega$ differ by a divisor of a non--zero function $f\in \mathbb C\left(\mathfrak R_\Gamma\right)$
$$
\mathrm{div}{(f\omega)}=\mathrm{div}{(f)}+ \mathrm{div}{(\omega)}.
$$
Different choices of $\omega$ have the same degree since $\deg{\left(\mathrm{div}{(f)}\right)}=0$.

For a divisor $\mathfrak a$, we let
$$
L(\mathfrak a)=\left\{f\in \mathbb C\left(\mathfrak R_\Gamma\right); \ \ f=0 \ \text{or} \ \mathrm{div}{(f)}+ \mathfrak a\ge 0\right\}.
$$
We have the following three facts:
\begin{itemize}
\item[(1)] for $\mathfrak a=0$, we have $L(\mathfrak a)=\mathbb C$;
\item[(2)] if $\deg{(\mathfrak a)}< 0$, then  $L(\mathfrak a)=0$;
  \item[(3)] the Riemann-Roch theorem: $\dim L(\mathfrak a)= \deg{(\mathfrak a)}-g(\Gamma)+1 + \dim L(K-\mathfrak a)$.
  \end{itemize}

Now, it is obvious that   $f\omega^m \in H^m\left(\mathfrak R_\Gamma\right)$ if and only if
$$
\mathrm{div}{(f\omega^m)}=\mathrm{div}{(f)}+m\mathrm{div}{(\omega)}= \mathrm{div}{(f)}+mK\ge 0.
$$
Equivalently, $f\in L(mK)$. Thus, we have that $\dim H^m\left(\mathfrak R_\Gamma\right)=  \dim L(mK)$. Finally, by the Riemann-Roch theorem, we have the following:
$$
\dim L(mK)=\deg{(mK)}-g(\Gamma)+1 + \dim L((1-m)K)=(2m-1)(g\left(\mathfrak R_\Gamma\right)-1)+ \dim L((1-m)K).
$$
Now, if $g(\Gamma)\ge 2 $, then  $\deg{(K)}=2(g(\Gamma)-1)>0$, and the claim easily follows from (1) and (2) above.
Next, assume that $g(\Gamma)=1$. If $\omega\in \dim H^1\left(\mathfrak R_\Gamma\right)$ s non--zero, then it has a degree zero.
Thus, it has no zeroes. This means that $\omega H^{l-1}\left(\mathfrak R_\Gamma\right)=H^l\left(\mathfrak R_\Gamma\right)$ for all $l\in \mathbb Z$. But since obviously $H^0\left(\mathfrak R_\Gamma\right)$ consists
of constants only, we obtain the claim. Finally, the case $g(\Gamma)=0$ is obvious from (2) since the degree of $mK$ is $2m(g(\Gamma)-1)<0$ for all $m\ge 1$.

\vskip .2in

Assume that  $g(\Gamma)\ge 1$ and $m\ge 1$. Then,  $\dim H^m\left(\mathfrak R_\Gamma\right)\neq 0$. Let $t= \dim H^m\left(\mathfrak R_\Gamma\right)$. We fix the basis
$\omega_1, \ldots, \omega_t$ of  $H^m\left(\mathfrak R_\Gamma\right)$. Let $z$ be any local coordinate on $\mathfrak R_\Gamma$. Then, locally there exists unique holomorphic functions
$\varphi_1, \ldots, \varphi_t$ such that $\omega_i=\varphi_i \left(dz\right)^m$, for all $i$. Then, again locally,
we can consider the Wronskian $W_z$ defined by
\begin{equation}\label{mhd-3}
W_z\left(\omega_1, \ldots, \omega_t\right)\overset{def}{=}\left|\begin{matrix}
\varphi_1(z) &
\cdots & \varphi_{t}(z) \\
\frac{d\varphi_1(z)}{dz} &
\cdots & \frac{d\varphi_{t}(z)}{dz} \\
&\cdots & \\
\frac{d^{t-1}\varphi_1(z)}{dz^{k-1}} &
\cdots & \frac{d^{t-1}\varphi_{t}(z)}{dz^{t-1}} \\
\end{matrix}\right|. 
\end{equation}

As proved in \cite[Proposition III.5.10]{FK}, collection of all
\begin{equation}\label{mhd-4}
W_z\left(\omega_1, \ldots, \omega_t\right)\left(dz\right)^{\frac{t}{2}\left(2m-1+t\right)} ,
\end{equation}
defines a non--zero holomorphic differential form
$$
W\left(\omega_1, \ldots, \omega_t\right)\in H^{\frac{t}{2}\left(2m-1+t\right)}\left(\mathfrak R_\Gamma\right).
$$
We call this form the Wronskian of the basis $\omega_1, \ldots, \omega_t$. It is obvious that a different choice of a basis of
$H^m\left(\mathfrak R_\Gamma\right)$ results in a Wronskian which differ from $W\left(\omega_1, \ldots, \omega_t\right)$ by a multiplication by
a non--zero complex number. Also,  the degree is given by
\begin{equation}\label{mhd-5}
\deg{\left(\mathrm{div}{\left(W\left(\omega_1, \ldots, \omega_t\right)\right)}\right)}= t\left(2m-1+t\right)(g\left(\mathfrak R_\Gamma\right)-1).
\end{equation}

\vskip .2in
Following \cite[III.5.9]{FK}, we make the following definition:

\begin{Def}\label{mhd-def}  Let $m\ge 1$ be an integer. We say that $\mathfrak a \in \mathfrak R_\Gamma$ is a $m$-Weierstrass point if
there exists a non--zero $\omega\in H^m\left(\mathfrak R_\Gamma\right)$ such that
$$
\nu_{\mathfrak a}(\omega)\ge \dim H^m\left(\mathfrak R_\Gamma\right).
$$
\end{Def}
Equivalently \cite[Proposition III.5.10]{FK} , if 
$$
\nu_{\mathfrak a}\left(W\left(\omega_1, \ldots, \omega_t\right)\right)\ge 1.
$$
When $m=1$ we speak about classical Wierstrass points. So, $1$-Weierstrass points are simply Weierstrass points.

\section{Interpretation in Terms of Modular Forms} \label{mhd-cont}

In this section we give interpretation of results of Section \ref{mhd} in terms of modular forms.
Again, $\Gamma$ stand for a Fuschsian group of the first kind. Let $m\ge 2$ be an even integer. We consider the space
$\mathcal A_m(\Gamma)$ be the space of all all meromorphic functions $f: \mathbb H\rightarrow \bbC$ 
such that $f(\gamma.z)=j(\gamma, z)^m f(z)$  ($z\in \mathbb H$, $\gamma\in \Gamma$) which are meromorphic at
every cusp for $\Gamma$. By \cite[Theorem 2.3.1]{Miyake}, there exists isomorphism of vector spaces $\mathcal A_m(\Gamma)\longrightarrow D^{m/2}\left(\mathfrak R_\Gamma\right)$,
denoted by $f\longmapsto \omega_f$ such that the following holds (see Section \ref{prelim} for notation, and \cite[Theorem 2.3.3]{Miyake}):

\begin{equation}\label{mhd-5}
  \begin{aligned}
  &\nu_{\mathfrak a_\xi}(f)=\nu_{\mathfrak a_\xi}(\omega_f) +\frac{m}{2}\left(1- \frac{1}{e_{\mathfrak a_\xi}}\right) \ \ \text{if} \ \xi\in \mathbb H\\
    &\nu_{\mathfrak a}(f)=\nu_{\mathfrak a}(\omega_f) + \frac{m}{2} \ \ \text{for $\Gamma$---cusp $\mathfrak a$.}\\
    & \mathrm{div}{(f)}=\mathrm{div}{(\omega_f)}+\sum_{\mathfrak a \in \mathfrak R_\Gamma} \frac{m}{2}\left(1- \frac{1}{e_{\mathfrak a}}\right) \mathfrak a,
     \end{aligned}
\end{equation}
where $1/e_{\mathfrak a}=0$ if $\mathfrak a$ is a cusp. Let $f\in M_m(\Gamma)$. Then, combining Lemma \ref{prelim-1} (vi) and (\ref{mhd-5}), we obtain

\begin{equation}\label{mhd-6}
\mathrm{div}{(\omega_f)}=\mathfrak c' _f -\sum_{\mathfrak a\in \mathfrak
R_\Gamma, \ elliptic} \left[\frac{m}{2}(1-1/e_{\mathfrak  a})\right]\mathfrak a - \frac{m}{2}\sum_{\mathfrak b\in \mathfrak
R_\Gamma, \ cusp}  \mathfrak b.
\end{equation}
This shows that $\omega_f$ is holomorphic everywhere except maybe at cusps and elliptic points.
Moreover, if $f\in  S_m(\Gamma)$, then (see Lemma \ref{prelim-1} (vii))
\begin{equation}\label{mhd-7}
\mathrm{div}{(\omega_f)}=\mathfrak c_f -  \sum_{\mathfrak a\in \mathfrak
R_\Gamma, \ elliptic} \left[\frac{m}{2}(1-1/e_{\mathfrak  a})\right]\mathfrak a - \left(\frac{m}{2} -1\right)\sum_{\mathfrak b\in \mathfrak
R_\Gamma, \ cusp}  \mathfrak b.
\end{equation}

Next, we determine all $f\in M_m(\Gamma)$ such that $\omega_f\in H^{m/2}\left(\mathfrak R_\Gamma\right)$. From (\ref{mhd-6}) we see that
such $f$ must belong to $S_m(\Gamma)$, and from (\ref{mhd-7})
\begin{equation}\label{mhd-8}
\mathfrak c_f\ge \sum_{\mathfrak a\in \mathfrak
R_\Gamma, \ elliptic} \left[\frac{m}{2}(1-1/e_{\mathfrak  a})\right]\mathfrak a+ \left(\frac{m}{2} -1\right)\sum_{\mathfrak b\in \mathfrak
  R_\Gamma, \ cusp}  \mathfrak b.
\end{equation}

Now, we define the subspace of $S_m(\Gamma)$ by
  $$
  S^H_m(\Gamma)=\left\{f\in S_m(\Gamma); \ \  \text{$f=0$ or $f$ satisfies (\ref{mhd-8})} \right\}.
  $$
  It is mapped via $f\longmapsto \omega_f$ isomorphically onto $H^{m/2}\left(\mathfrak R_\Gamma\right)$.
  
  We remark that when $m=2$,  (\ref{mhd-8}) and reduces to obvious $\mathfrak c_f\ge 0$. Hence,
  $S^H_2(\Gamma)= S_2(\Gamma)$ recovering  the standard
  isomorphism of $S_2(\Gamma)$ and $H^1(\mathfrak R_\Gamma)$ (see \cite[Theorem 2.3.2]{Miyake}). We have the following result:

  \begin{Lem} \label{mhd-9} Assume that $m, n\ge 2$  are even integers. Let $\Gamma$ be a Fuchsian group of the first kind. Then, we have  the following:
    \begin{itemize}
    \item[(i)] $S^H_2(\Gamma)= S_2(\Gamma)$.
    \item[(ii)]  $S^H_m(\Gamma)$ is isomorphic to  $H^{m/2}\left(\mathfrak R_\Gamma\right)$.
    \item[(iii)]  $S^H_m(\Gamma)=\{0\}$ if $g(\Gamma)=0$.
    \item[(iv)] Assume that $g(\Gamma)=1$.  Let us write $S_2(\Gamma)=\mathbb C \cdot f$, for some non--zero cuspidal form $f$. Then, we have
      $S^H_m(\Gamma)= \mathbb C \cdot f^{m/2}$. 
    \item[(v)] $\dim{S^H_m(\Gamma)}= (m-1)\left(g(\Gamma)-1\right)$ if $g(\Gamma)\ge 2$.
    \item[(vi)] $S^H_m(\Gamma)\cdot S^H_n(\Gamma)\subset S^H_{m+n}(\Gamma)$.
       \item[(vii)] There are no  $m/2$--Weierstrass points on $\mathfrak R_\Gamma$ for $g(\Gamma)\in \{0, 1\}$. 
    \item[(viii)] Assume that $g(\Gamma)\ge 2$, and $\mathfrak a_\infty$ is a $\Gamma$-cusp.  Then,  $\mathfrak a_\infty$ is a 
      $\frac{m}{2}$--Weierstrass point if and only if there exists $f\in  S^H_m(\Gamma)$, $f\neq 0$, such that
      $$
      \mathfrak c'_f(\mathfrak a_\infty)\ge \begin{cases}   \frac{m}{2} + g(\Gamma) \ \ \text{if} \ \ m=2;\\
        \frac{m}{2} +  (m-1)(g(\Gamma)-1)  \ \ \text{if} \ \ m\ge 4.\\
        \end{cases}
      $$
    \item[(ix)] Assume that $g(\Gamma)\ge 1$, and $\mathfrak a_\infty$ is a $\Gamma$-cusp.  Then, there exists a basis $f_1, \ldots f_t$ of $S_m^H(\Gamma)$ such that
      their $q$--expansions are of the form
      $$
      f_u=a_uq^{i_u}+ \text{higher order terms in $q$}, \ \ 1\le u\le t,
      $$
      where 
      $$
      \frac{m}{2}\le i_1< i_2< \cdots < i_t \le \frac{m}{2}+ m\left(g(\Gamma)-1\right),
      $$
      and
      $$
      a_u\in \mathbb C, \ \  a_u\neq 0.
      $$
      \item[(x)] Assume that $g(\Gamma)\ge 1$, and $\mathfrak a_\infty$ is a $\Gamma$-cusp. Then,  $\mathfrak a_\infty$ is {\bf not} a 
      $\frac{m}{2}$--Weierstrass point if and only if there exists a basis $f_1, \ldots f_t$ of $S_m^H(\Gamma)$ such that
      their $q$--expansions are of the form
      $$
      f_u=a_uq^{u+m/2-1}+ \text{higher order terms in $q$}, \ \ 1\le u\le t,
      $$
      where 
      $$
      a_u\in \mathbb C, \ \  a_u\neq 0.
      $$ 
      \end{itemize}  
    \end{Lem} 
  \begin{proof} (i) and (ii) follow from above discussion. Next, using above discussion and (\ref{mhd-2}) we obtain
    $$
    \dim{S^H_m(\Gamma)} =\dim H^{m/2}\left(\mathfrak R_\Gamma\right)=\begin{cases}  0 \ \ &\text{if} \ \ m\ge 2, g(\Gamma)=0;\\
     g(\Gamma)    \ \ &\text{if} \ \ m=2, \ g(\Gamma)\ge 1;\\
   g(\Gamma)    \ \ & \text{if} \ \ m\ge 4, \ g(\Gamma)= 1;\\
     (m-1)\left(g(\Gamma)-1\right)    \ \ &\text{if} \ \ m\ge 4, \ g(\Gamma)\ge 2.\\           
    \end{cases}
    $$
    This immediately implies (iii) and (v). Next, assume that $g(\Gamma)=1$. Then, we see that $\dim{S^H_m(\Gamma)}\le 1$ for all even integers $m\ge 4$.
    It is well known that $f^{m/2} \in S_m(\Gamma)$. Next, (\ref{mhd-8}) for $m=2$ implies $\mathrm{div}{(\omega_f)}=\mathfrak c_f$. Also, the degree of $\mathfrak c_f$ is zero by
    Lemma \ref{prelim-1} (vii). Hence,
    $$
    \mathrm{div}{(\omega_f)}=\mathfrak c_f=0.
    $$
    Using \cite[Theorem 2.3.2]{Miyake}, we obtain
    $$
    \omega_{f^{m/2}}=\omega^{m/2}_f.
    $$
    Hence
    $$
    \mathrm{div}{(\omega_{f^{m/2}})}=\frac{m}{2} \mathrm{div}{(\omega_f)}=0.
    $$
    Then, applying (\ref{mhd-7}) with $f^{m/2}$ in place of $f$, we obtain
    $$
    \mathfrak c_{f^{m/2}} = \sum_{\mathfrak a\in \mathfrak
R_\Gamma, \ elliptic} \left[\frac{m}{2}(1-1/e_{\mathfrak  a})\right]\mathfrak a + \left(\frac{m}{2} -1\right)\sum_{\mathfrak b\in \mathfrak
      R_\Gamma, \ cusp}  \mathfrak b.
    $$
    This shows that $f^{m/2}\in S^H_m(\Gamma)$ proving (iv). Finally, (vi) follows from \cite[Theorem 2.3.1]{Miyake}. We can also see that directly as follows. Let
    $0\neq f\in  S^H_m(\Gamma)$ and $0\neq g\in  S^H_n(\Gamma)$. Then, $fg\in S_{m+n}(\Gamma)$ since  $f\in  S^H_m(\Gamma)\subset  S_m(\Gamma)$ and $g\in  S^H_n(\Gamma)\subset  S_n(\Gamma)$.
    We have the following:
    $$
    \mathrm{div}{(f\cdot g)}= \mathrm{div}{(f)}+\mathrm{div}{(g)}.
    $$
    Using Lemma \ref{prelim-1} (vi) we can rewrite this identity as follows:
    \begin{align*}
     &\mathfrak c'_{f\cdot g} - \sum_{\mathfrak a\in \mathfrak
       R_\Gamma, \ \ elliptic}  \left[\frac{m+n}{2}(1-1/e_{\mathfrak a})\right] \mathfrak a= \\
&\mathfrak c'_{f} - \sum_{\mathfrak a\in \mathfrak
  R_\Gamma, \ \ elliptic}  \left[\frac{m}{2}(1-1/e_{\mathfrak a})\right] \mathfrak a + \mathfrak c'_{g} - \sum_{\mathfrak a\in \mathfrak
  R_\Gamma, \ \ elliptic}  \left[\frac{n}{2}(1-1/e_{\mathfrak a})\right] \mathfrak a.
    \end{align*}
    By Lemma  \ref{prelim-1} (vii) we obtain:
\begin{align*}
     &\mathfrak c_{f\cdot g} - \sum_{\mathfrak a\in \mathfrak
       R_\Gamma, \ \ elliptic}  \left[\frac{m+n}{2}(1-1/e_{\mathfrak a})\right] \mathfrak a  -  \left(\frac{m+n}{2} -1\right)\sum_{\mathfrak b\in \mathfrak
  R_\Gamma, \ cusp}  \mathfrak b=\\
&\left(\mathfrak c_{f} - \sum_{\mathfrak a\in \mathfrak
    R_\Gamma, \ \ elliptic}  \left[\frac{m}{2}(1-1/e_{\mathfrak a})\right] \mathfrak a
 -  \left(\frac{m}{2} -1\right)\sum_{\mathfrak b\in \mathfrak
  R_\Gamma, \ cusp}  \mathfrak b\right)+\\
&\left(\mathfrak c_{g} - \sum_{\mathfrak a\in \mathfrak
   R_\Gamma, \ \ elliptic}  \left[\frac{n}{2}(1-1/e_{\mathfrak a})\right] \mathfrak a
  -  \left(\frac{n}{2} -1\right)\sum_{\mathfrak b\in \mathfrak
  R_\Gamma, \ cusp}  \mathfrak b\right).
    \end{align*}
Finally, (vi) follows applying (\ref{mhd-8})  since both terms on the right hand of equality  are $\ge 0$.

Next, (vii) follows immediately form the discussion in Section \ref{mhd}, and it is well--known. (viii) is a reinterpretation of
Definition \ref{mhd-def}. The details are left to the reader as an easy exercise.

Finally, (ix) and (x) in the case of $g(\Gamma)=1$ are obvious since by Lemma \ref{prelim-1}  we have $S_2(\Gamma)= \mathbb C \cdot f$ where
$$
\mathfrak c'_f=\mathfrak a_\infty + \sum_{\substack{\mathfrak b \in \mathfrak
    R_\Gamma, \ cusp\\
    \mathfrak b \neq \mathfrak a_\infty}} \mathfrak b.
$$

We prove  (ix) and (x) in the case of $g(\Gamma)\ge 2$. Let $f\in S_m^H(\Gamma)$, $f\neq 0$. Then, by the definition of $S_m^H(\Gamma)$, we obtain
\begin{equation} \label{mhd-11}
  \mathfrak c'_f(\mathfrak a_\infty)=1+ \mathfrak c_f(\mathfrak a_\infty)\ge 1+ \left(\frac{m}{2}-1\right)=\frac{m}{2}.
  \end{equation}
On the other hand, again by the definition of $S_m^H(\Gamma)$ (see (\ref{mhd-8})) and the fact that $\mathfrak c'_f\ge 0$, we obtain

\begin{align*}
  \deg{(\mathfrak c'_f)} & =\sum_{\mathfrak a \in \mathfrak R_\Gamma} \ \mathfrak c'_f(\mathfrak a) \ge \\
                        &\sum_{\mathfrak a\in \mathfrak R_\Gamma, \ elliptic}  \mathfrak c'_f(\mathfrak a) +  \sum_{\substack{\mathfrak b \in \mathfrak
      R_\Gamma, \ cusp\\    \mathfrak b \neq \mathfrak a_\infty}} \mathfrak c'_f(\mathfrak b) + \mathfrak c'_f(\mathfrak a_\infty)\ge\\
   &\sum_{\mathfrak a\in \mathfrak R_\Gamma, \ elliptic} \left[\frac{m}{2}(1-1/e_{\mathfrak  a})\right]+ \frac{m}{2} \left(t-1\right) + \mathfrak c'_f(\mathfrak a_\infty)
  \end{align*}
where $t$ is the number of inequivalent $\Gamma$--cusps. The degree $\deg{(\mathfrak c'_f)}$ is given by Lemma \ref{prelim-1} (vi)  

\begin{align*}
  \deg{(\mathfrak c'_f)} &=   \dim M_m(\Gamma)+ g(\Gamma)-1\\
  &=   \begin{cases} 2(g(\Gamma)-1)+t  \ \ \text{if} \ \ m=2;\\
                 m(g(\Gamma)-1)+ \frac{m}{2}t+ \sum_{\substack{\mathfrak a\in \mathfrak
        R_\Gamma, \\ elliptic}} \left[\frac{m}{2}(1-1/e_{\mathfrak a})\right] \ \ \text{if} \ \ m\ge 4.\\
    \end{cases}
\end{align*}
Combining with the previous inequality, we obtain
$$
\mathfrak c'_f(\mathfrak a_\infty) \le \frac{m}{2}+ m(g(\Gamma)-1) \ \ \text{if} \ \ m\ge 2.
$$
Having in mind (\ref{mhd-11}), the rest of (ix) has standard argument (see for example \cite[Lemma 4.3]{Muic}).  Finally, (x) follows (viii) and (ix). 
\end{proof}

\vskip .2in   
The criterion in Lemma \ref{mhd-9} (x) is a quite good criterion to check whether or not $\mathfrak a_\infty$ is a Weierstrass points (the case $m=2$) using computer systems such as SAGE
since we need just to list the basis. This case is well-known (see \cite[Definition 6.1]{ono}). This criterion has been used in practical computations in
combination with SAGE in \cite{MuicKodrnja} for $\Gamma=\Gamma_0(N)$.

\vskip .2in

But it is not good when $m\ge 4$, regarding the bound for  $S^H_m(\Gamma)$ given by Lemma \ref{mhd-9} (ix), 
since then a basis  of $S_m(\Gamma)$ contains properly normalized cusp forms having leading terms $q^{m/2}, \ldots,  q^{m/2+ m(g(\Gamma)-1)}$ at least when $\Gamma$ has elliptic
points for $m$ large enough and we do not know which of then belong to $S^H_m(\Gamma)$.
We explain that in Corollary \ref{mhd-14} bellow.

\vskip .2in 
First, we recall the following result \cite[Lemma 2.9]{Muic2} which is well-known in a slightly different notation (\cite{pete-1}, \cite{pete-2}):

\begin{Lem}\label{mhd-12} Let $m\ge 4$ be an even integer such that $\dim S_m(\Gamma)\ge g(\Gamma)+1$.
Then, for all 
$1\le i \le t_m-g$, there exists $f_i\in S_m(\Gamma)$ such that 
$\mathfrak c'_{f_i}(\mathfrak a_\infty)= i$.
\end{Lem}

\vskip .2in

\begin{Lem}\label{mhd-13} Assume that $\Gamma$ has elliptic points. (For example, $\Gamma=\Gamma_0(N)$.) Then, for a sufficiently large even integer  $m$, we have
  \begin{equation}\label{mhd-15}
  \frac{m}{2} +  m(g(\Gamma)-1)  \le \dim S_m(\Gamma) -g(\Gamma).
  \end{equation}
\end{Lem}
\begin{proof} Assume that $m\ge 4$ is an even integer. Then, by Lemma \ref{prelim-1} (v), we obtain

  \begin{align*}
    &\dim S_m(\Gamma) -g(\Gamma) - \left(\frac{m}{2} +  m(g(\Gamma)-1)\right) \\
    &= \left(\frac{m}{2}-1\right)t - \frac{m}{2} + \sum_{\substack{\mathfrak a\in \mathfrak
        R_\Gamma, \\ elliptic}} \left[\frac{m}{2}(1-1/e_{\mathfrak a})\right] -2g(\Gamma)+1\\
    &\ge  \sum_{\substack{\mathfrak a\in \mathfrak
        R_\Gamma, \\ elliptic}} \left[\frac{m}{2}(1-1/e_{\mathfrak a})\right] -2g(\Gamma)-t+1.
\end{align*}                        
Since $\Gamma$ has elliptic points, the last term is $\ge 0$ for $m$ sufficiently large even integer.
\end{proof}

\vskip .2in
\begin{Cor}\label{mhd-14} Assume that (\ref{mhd-15}) holds. Then, given a basis $f_1, \ldots f_t$ of $S_m^H(\Gamma)$ such that $\mathfrak c'_{f_j}(\mathfrak a_\infty)=i_j$, $1\le j\le t$, 
 where 
      $$
      \frac{m}{2}\le i_1< i_2< \cdots < i_t \le \frac{m}{2}+ m\left(g(\Gamma)-1\right)
      $$
      can be extended by additional $g(\Gamma)$ cuspidal modular forms in $S_m(\Gamma)$ to obtain the collection $F_k$, $\frac{m}{2}\le k \le \frac{m}{2}+ m\left(g(\Gamma)-1\right)$
      such that $\mathfrak c'_{F_k}(\mathfrak a_\infty)=k$ for all $k$.
      \end{Cor}
\begin{proof} This follows directly from Lemmas \ref{mhd-12} and \ref{mhd-13}. 
  \end{proof}

\vskip .2in
Now, explain the algorithm for testing that  $\mathfrak a_\infty$ is  a 
$\frac{m}{2}$--Weierstrass point for $m\ge 6$.  It requires some geometry.
We recall that $\mathfrak R_\Gamma$ is hyperelliptic if $g(\Gamma)\ge 2$, and there is a degree two map onto $\mathbb P^1$.
By general theory \cite[Chapter VII, Proposition 1.10]{Miranda}, if $g(\Gamma)=2$, then $\mathfrak R_\Gamma$ is hyperelliptic.
If $\mathfrak R_\Gamma$ is not hyperelliptic, then $\dim S_2(\Gamma)= g(\Gamma)\ge 3$, and  the regular  map
$\mathfrak R_\Gamma\longrightarrow \mathbb P^{g(\Gamma)-1}$ attached to a canonical divisor $K$  is an isomorphism  onto its image
\cite[Chapter VII, Proposition 2.1]{Miranda}.

\vskip .2in

Let $\Gamma=\Gamma_0(N)$, $N\ge 1$. Put $X_0(N)=\mathfrak R_{\Gamma_0(N)}$. We recall that  $g(\Gamma_0(N))\ge 2$ unless
$$
\begin{cases}
N\in\{1-10, 12, 13, 16, 18, 25\} \ \ \text{when $g(\Gamma_0(N))=0$, and}\\
N\in\{11, 14, 15, 17, 19-21, 24, 27, 32, 36, 49\} \ \ \text{when
	$g(\Gamma_0(N))=1$.}
\end{cases}
$$
Let $g(\Gamma_0(N))\ge 2$. Then,  we remark that Ogg \cite{Ogg} has determined all $X_0(N)$ which are hyperelliptic curves.
In view of Ogg's paper, we see that $X_0(N)$ is {\bf not hyperelliptic} for
$N\in \{34, 38, 42, 43, 44, 45, 51-58, 60-70\}$ or $N\ge 72$. This implies $g(\Gamma_0(N))\ge 3$.

\vskip .2in
Before we begin the study of spaces $S_m^H(\Gamma)$ we give the following lemma. 

\vskip .2in
\begin{Lem}\label{cts-5000}  Let $m\ge 4$ be an even integer. Let us select a basis $f_0, \ldots, f_{g-1}$, $g=g(\Gamma)$, of $S_2(\Gamma)$. Then,
  all of $\binom{g+\frac{m}{2}-1}{\frac{m}{2}}$ monomials  $f_0^{\alpha_0}f_1^{\alpha_1}\cdots f_{g-1}^{\alpha_{g-1}}$, $\alpha_i\in \mathbb Z_{\ge 0}$, $\sum_{i=0}^{g-1} \alpha_i=\frac{m}{2}$,
  belong to $S_m^H(\Gamma)$. We denote this subspace of $S_m^H(\Gamma)$ by $S_{m, 2}^H(\Gamma)$. 
  \end{Lem}
\begin{proof} This follows from Lemma \ref{mhd-9} (vi) since $S_2(\Gamma)=S_2^H(\Gamma)$ (see Lemma \ref{mhd-9} (i)).
 \end{proof}

 \vskip .2in
 \begin{Thm}\label{cts-50000}  Let $m\ge 4$ be an even integer. Assume that  $\mathfrak R_\Gamma$ is not hyperelliptic. Then, we have
   $$
   S_{m, 2}^H(\Gamma)= S_m^H(\Gamma).
   $$
 \end{Thm}
 \begin{proof} We use notation of Section \ref{mhd} freely. The reader should review Lemma \ref{mhd-9}.
    Let $F\in S_2(\Gamma)$, $F\neq 0$.  We define a holomorphic differential form $\omega \in H\left(\mathfrak R_\Gamma\right)$
    by $\omega=\omega_F$. Define a canonical class $K$ by $K=\mathrm{div}{(\omega)}$. We prove the following:

    \begin{equation}\label{cts-2}
  L\left(\frac{m}{2}K\right)=\left\{\frac{f}{F^{m/2}}; \ \ f\in S_m^H(\Gamma) \right\}.
    \end{equation}
The case $m=2$ is of course well--known.   By  the Riemann-Roch theorem and standard results recalled in Section \ref{mhd} we have
  \begin{align*}
    \dim L\left(\frac{m}{2}K\right) &= \deg{\left(\frac{m}{2}K\right)}-g(\Gamma)+1 + \dim L\left(\left(1- \frac{m}{2}\right)K\right)\\
    &=   (m-1)(g(\Gamma)-1)+\begin{cases} 1 \ \ \text{if} \ \ m=2;\\
                                          0 \ \ \text{if} \ \ m\ge 4.\\
                             \end{cases}           
     \end{align*}

Next,  we recall that $S_2(\Gamma)=S_2^H(\Gamma)$  (see Lemma \ref{mhd-9} (i)). Then, Lemma \ref{mhd-9} (vi) we obtain 
$F^{\frac{m}{2}}\in S_m^H(\Gamma)$. Therefore,  $f/F^{\frac{m}{2}}\in  \mathbb C\left(\mathfrak R_\Gamma\right)$ for all $f\in S_m^H(\Gamma)$.

By the correspondence described in (\ref{mhd-5}) we have

\begin{align*}
\mathrm{div}{(F)}=\mathrm{div}{(\omega_F)}+\sum_{\mathfrak a \in \mathfrak R_\Gamma} \left(1- \frac{1}{e_{\mathfrak a}}\right) \mathfrak a &=
K+ \sum_{\mathfrak a \in \mathfrak R_\Gamma} \left(1- \frac{1}{e_{\mathfrak a}}\right) \mathfrak a\\
&= K+  \sum_{\mathfrak a\in \mathfrak R_\Gamma, \ \ elliptic} (1-1/e_{\mathfrak a})  \mathfrak a + \sum_{\substack{\mathfrak b \in \mathfrak R_\Gamma, \\ cusp}} \mathfrak b.
\end{align*}

Thus, for  $f\in S_m^H(\Gamma)$, we have the following:
\begin{align*}
  \mathrm{div}{\left(\frac{f}{F^{\frac{m}{2}}}\right)} + \frac{m}{2}K &= \mathrm{div}{(f)}- \frac{m}{2}\mathrm{div}{(F)}  + \frac{m}{2}K\\
  &= \mathrm{div}{(f)}-\frac{m}{2}\sum_{\mathfrak a\in \mathfrak R_\Gamma, \ \ elliptic} (1-1/e_{\mathfrak a}) \mathfrak a-\frac{m}{2}
  \sum_{\substack{\mathfrak b \in \mathfrak R_\Gamma, \\ cusp}} \mathfrak b
\end{align*}
Next, using Lemma \ref{prelim-1} (vi), the right--hand side becomes
$$
\mathfrak c'_f - \sum_{\mathfrak a\in \mathfrak
R_\Gamma, \ \ elliptic}  \left[\frac{m}{2}(1-1/e_{\mathfrak  a})\right]\mathfrak a - \frac{m}{2}
  \sum_{\substack{\mathfrak b \in \mathfrak R_\Gamma, \\ cusp}} \mathfrak b \ge 0
$$
by the definition of $S_m^H(\Gamma)$. Hence,  $f/F^{\frac{m}{2}}\in L\left(\frac{m}{2}K\right)$. Now, comparing the dimensions of
the right-hand and left-hand side in (\ref{cts-2}), we obtain their equality. This proves (\ref{cts-2}).

\vskip .2in
Let $W$ be any finite dimensional $\mathbb C$--vector space. Let $\text{Symm}^k(W)$ be symmetric tensors of degree $k\ge 1$.
Then, by Max Noether theorem (\cite{Miranda}, Chapter VII, Corollary 3.27) the multiplication induces surjective map
$$
\text{Symm}^k{\left(L\left(K\right)\right)}\twoheadrightarrow L\left(\frac{m}{2}K\right).
$$
The theorem follows.
\end{proof}

 \vskip .2in
 Now, we combine Theorem \ref{cts-50000} with Lemma \ref{mhd-9} (x) to obtain a good criterion in the case $m\ge 4$
 for {\bf testing} that  $\mathfrak a_\infty$ is  a  $\frac{m}{2}$--Weierstrass point. We give examples in Section \ref{cts} (see
 Propositions \ref{cts-5001} and \ref{cts-5002}). 

 \vskip .2in
 \begin{Cor}\label{cts-500000}  Let $m\ge 4$ be an even integer. Assume that  $\mathfrak R_\Gamma$ is not hyperelliptic. Assume that $\mathfrak a_\infty$ is a cusp for $\Gamma$.
   Let us select a basis $f_0, \ldots, f_{g-1}$, $g=g(\Gamma)$, of $S_2(\Gamma)$. Compute $q$--expansions of all monomials
   $$
   f_0^{\alpha_0}f_1^{\alpha_1}\cdots f_{g-1}^{\alpha_{g-1}}, \ \  \alpha_i\in \mathbb Z_{\ge 0}, \ \sum_{i=0}^{g-1} \alpha_i=\frac{m}{2}.
   $$
   Then,  $\mathfrak a_\infty$ is {\bf not} a   $\frac{m}{2}$--Weierstrass point if and only if there exist a basis of the space of all such monomials,
   $F_1, \ldots F_t$, $t=\dim{S_m^H(\Gamma)}=(m-1)(g-1)$
   (see Lemma \ref{mhd-9} (v)), such that their $q$--expansions are of the form
      $$
      F_u=a_uq^{u+m/2-1}+ \text{higher order terms in $q$}, \ \ 1\le u\le t,
      $$
      where 
      $$
      a_u\in \mathbb C, \ \  a_u\neq 0.
      $$
\end{Cor}

 \vskip .2in
When $\mathfrak R_\Gamma$ is hyperelliptic, for example if $g(\Gamma)=2$, the space $S_{m, 2}^H(\Gamma)$ could be proper subspace of $S_{m}^H(\Gamma)$.
For example, if $N=35$, then $g(\Gamma_0(N))=3$ and $X_0(N)$ is hyperelliptic. For $m=4, 6, 8, 8, 10, 12, 14$ we checked that $\dim{S_{m, 2}^H(\Gamma)}= m+1$ while
by general theory $\dim{S_{m}^H(\Gamma)}= (m-1)(g(\Gamma_0(N))-1)= 2(m-1)$. We see that
$$
\dim{S_{m}^H(\Gamma)} - \dim{S_{m, 2}^H(\Gamma)}=m-3\ge 1, \ \ \text{for} \ m=4, 6, 8, 8, 10, 12, 14.
$$

\vskip .2in
In fact, the case of $g(\Gamma)= 2$ could be covered in full generality. We leave easy proof of the following proposition to the reader. 

\begin{Prop}\label{cts-5002} Assume that $g(\Gamma)= 2$. Let $f_0, f_1$ be a basis of $S_2(\Gamma)$. Then, for any even integer $m\ge 4$,
  $f_0^uf_1^{\frac{m}{2}-u}$, $0\le u \le m$ is a basis of $S_{m, 2}^H(\Gamma)$. Therefore,
  $$
  \dim{S_{m}^H(\Gamma)}= (m-1)(g(\Gamma)-1)= m-1> \frac{m}{2}+1, \ \ \text{for} \ m\ge 6,
  $$
  and  $S_{4, 2}^H(\Gamma)= S_{4}^H(\Gamma)$. 
\end{Prop}

\vskip .2in
We end this section with the standard yoga.
\vskip .2in

  \begin{Thm} \label{mhd-10} Let $m\ge 2$ be an even integer. Let $\Gamma$ be a Fuchsian group of the first kind such that $g(\Gamma)\ge 1$. Let
    $t= \dim S^H_m(\Gamma)=\dim  H^{m/2}\left(\mathfrak R_\Gamma\right)$.
    Let us fix a basis $f_1, \ldots, f_t$ of $S^H_m(\Gamma)$, and let $\omega_1, \ldots, \omega_t$ be the corresponding basis of $H^{m/2}\left(\mathfrak R_\Gamma\right)$. As above,
    we construct
    holomorphic differential $W\left(\omega_1, \ldots, \omega_t\right)\in H^{\frac{t}{2}\left(m-1+t\right)}\left(\mathfrak R_\Gamma\right)$. We also construct the Wronskian $W(f_1,\ldots, f_t)\in
      S_{t(m+t-1)}(\Gamma)$  (see Proposition \ref{wron-2}). Then, we have the following equality:
      $$
    \omega_{W(f_1,\ldots, f_t)}=W\left(\omega_1, \ldots, \omega_t\right).
    $$
    In particular, we obtain the following:
    $$
    W(f_1,\ldots, f_t)\in S^H_{t(m+t-1)}(\Gamma).
    $$
    Moreover, assume that $\mathfrak a_\infty$ is a $\Gamma$-cusp. Then, $\mathfrak a_\infty$ is a 
    $\frac{m}{2}$--Weierstrass point if and only if
    $$
    \mathfrak c'_{ W(f_1,\ldots, f_t)}(\mathfrak a_\infty)\ge 1+ \frac{t}{2}\left(m-1+t\right) \ \ \text{i.e.,} \ \
    \mathfrak c_{ W(f_1,\ldots, f_t)}(\mathfrak a_\infty)\ge \frac{t}{2}\left(m-1+t\right)
    $$
    (See also Lemma \ref{mhd-9} (x) for more effective formulation of the criterion.)
    \end{Thm}
  \begin{proof} Since this is a equality of two meromorphic differentials, it is enough to check the identity locally. Let $z\in \mathbb H$ be a non--elliptic point, and
    $z\in U\subset \mathbb H$ a chart of
    $\mathfrak a_z$ such that $U$ does not contain any elliptic point. Then, one can use \cite[Section 2.3]{Miyake} to check the equality directly. Indeed, we have the following argument.

Let $z_0\in \mathbb H$ be a non--elliptic point, and $z\in U\subset \mathbb H$ a chart of
$\mathfrak a_{z_0}$ such that $U$ does not contain any elliptic point. 

Let $t_\fraka$ be a local coordinate on a neighborhood $V_\fraka =\pi(U)$. By the \cite[Section 1.8]{Miyake} if $U$ is small enough, projection $\pi$ gives homeomorphism of $U$ to
$V_\fraka$ such that
\begin{equation}\label{mhd-16}
t_\fraka \circ \pi(z)=z \ \text{ for } z \in U.
\end{equation}

Let $f \in \mathcal A_m(\Gamma)$ and let $\omega_f$ be the corresponding differential. Locally there exist unique meromorphic function $\varphi$ such that
$\omega_f=\varphi \left(dz\right)^{m/2}$.

By \cite[Section 2.3]{Miyake}, local correspondence $f\mapsto \omega_f$ is given by
$$\varphi (t_\fraka \circ \pi(z)) = f(z) \left( d(t_\fraka\circ\pi)/dz \right)^{-m/2},$$
which by the choice of local chart \ref{mhd-16} become
$$\varphi(z)=f(z) \ \text{ for } \ z \in U . $$ So, in the neighborhood of non-elliptic point $z\in \bbH$ we have 
$$\omega_f=f \left(dz\right)^{m/2}.$$
This gives us local identity
$$W_z\left(\omega_1, \ldots, \omega_t\right) = W(f_1,\ldots, f_t).$$
Since above is valid for any even $m\ge 2$, we get local identity of two meromorphic differentials
\begin{align*}
\omega_{W(f_1,\ldots, f_t)} &= W(f_1,\ldots, f_t) \left(dz\right)^{\frac{t}{2}\left(m-1+t\right)} \\ 
&= W_z\left(\omega_1, \ldots, \omega_t\right)\left(dz\right)^{\frac{t}{2}\left(m-1+t\right)} \\
&=  W\left(\omega_1, \ldots, \omega_t\right)
\end{align*}

 Now, assume that $\mathfrak a_\infty$ is a $\Gamma$-cusp. Then, $\mathfrak a_\infty$ is a 
 $\frac{m}{2}$--Weierstrass point if and only if
 $$
 \nu_{\mathfrak a_\infty}\left(W\left(\omega_1, \ldots, \omega_t\right)\right)\ge 1
 $$
 i.e., 
$$
 \nu_{\mathfrak a_\infty}\left( \omega_{W(f_1,\ldots, f_t)}\right) \ge 1
 $$
 by the first part of the proof. Finally, by (\ref{mhd-5}), this is equivalent to 
 $$
 \mathfrak c'_{ W(f_1,\ldots, f_t)}(\mathfrak a_\infty)=  \nu_{\mathfrak a_\infty}\left(\omega_{W(f_1,\ldots, f_t)}\right) + \frac{t}{2}\left(m-1+t\right)\ge 1+ \frac{t}{2}\left(m-1+t\right).
 $$
 This completes the proof of the theorem.
\end{proof}

\section{Explicit computations based on Corollary \ref{cts-500000} for $\Gamma=\Gamma_0(N)$} \label{cts}

In this section we  apply the algorithm in Corollary \ref{cts-500000} combined with SAGE. The method is the following. We take $q$-expansions of the base elements of $S_2(\Gamma_0(N))$:
$$
f_0, \ldots, f_{g-1},
$$
where $g=g(\Gamma_0(N))$. For even $m \ge 4$, we compute $q$--expansions of  all monomials of degree $m/2$:
$$
   f_0^{\alpha_0}f_1^{\alpha_1}\cdots f_{g-1}^{\alpha_{g-1}}, \ \  \alpha_i\in \mathbb Z_{\ge 0}, \ \sum_{i=0}^{g-1} \alpha_i=\frac{m}{2}.
$$
   The number of monomials is
   $$
   \binom{g+m/2-1}{m/2}
   $$

\vskip .2in
By selecting first $m/2+m \cdot (g-1)$ terms from $q$--expansions of the monomials (see Lemma \ref{mhd-9} (ix)), we can create matrix of size  
   $$
   \binom{g+m/2-1}{m/2} \times \left(\frac{m}{2}+m \cdot (g-1)\right).
   $$ 

  \vskip .2in
Then, we perform suitable integral Gaussian elimination method to transform the matrix into row echelon form.
The procedure is as follows. We successively sort and transform the row matrices to cancel the leading row coefficients with the same number of leading zeros as their predecessor.
We use the {\it Quicksort algorithm} for sorting. We obtain  the transformed matrix and the transformation matrix.
The non-null rows of the transformed matrix give the $q$-expansions of the basis elements, and the corresponding rows of the transformation matrix give the corresponding linear combinations of monomials.
\vskip .2in

Using above described method we perform various computations mentioned below. For example,  we can easily verify particular cases
Theorem \ref{cts-50000}.

\begin{Prop}\label{cts-5001} For $m=4, 6, 8, 10, 12$ and $N=34, 38, 44, 55$, and for  $m=4, 6, 8, 10$ and $N=54, 60$,  we have $S_{m}^H(\Gamma_0(N))=S_{m, 2}^H(\Gamma_0(N))$.
(We remark that all curves $X_0(N)$, $N\in \{34, 38, 44, 54, 55, 60\}$ are not hyperelliptic (see the paragraph after Corollary \ref{mhd-14}.)  
\end{Prop}

\vskip .2in
We can also deal with generalized Weierstrass points. For example, we can check the following result:

\begin{Prop}\label{cts-5002} For $m=2, 4, 6, 8, 10$, $\mathfrak a_\infty$ is not $\frac{m}{2}$--Weierstrass point for $X_0(34)$.
  Next,  $\mathfrak a_\infty$ is not ($1$--)Weierstrass point for $X_0(55)$, but it is $\frac{m}{2}$--Weierstrass point for $X_0(55)$ and $m=4, 6, 8, 10$.
\end{Prop}

\vskip .2in

For example,  let $m=4$. Then, for $X_0(34)$  the monomials are
\begin{align*}
f_0^2 & =q^{2}-4q^{5}-4q^{6}+12q^{8}+12q^{9}-2q^{10} \\
f_0f_1& = q^{3}-q^{5}-2q^{6}-2q^{7}+2q^{8}+5q^{9}+2q^{10} \\
f_0f_2 & = q^{4}-2q^{5}-q^{6}-q^{7}+6q^{8}+6q^{9}+2q^{10} \\
 -f_1^2 + f_0f_2& = -2q^{5}+q^{6}-q^{7}+5q^{8}+6q^{9}+4q^{10} \\
-f_1^2 + f_0f_2 + 2f_1f_2& =-3q^{6}-5q^{7}+11q^{8}+16q^{9}+2q^{10} \\
-f_1^2 + f_0f_2 + 2f_1f_2 + 3f_2^2 &= -17q^{7}+17q^{8}+34q^{9}+17q^{10} \\
\end{align*}
Their first exponents  are $\frac{m}{2}=4, 5, 6, \frac{m}{2}+ (m-1)(g-1)-1=7$ which shows that
$\mathfrak a_\infty$ is not $2$--Weierstrass point for $X_0(34)$.

\vskip .2in
For $X_0(55)$  the monomials are
\begin{align*}
 & f_0^2\\
 & f_0f_1\\
 & f_0f_2\\
 & f_0f_3\\
 & f_0f_4 \\
 & -f_1f_2 + f_0f_3\\
 & -f_1f_2 + f_0f_3 + 2f_2f_3\\
 & -f_1f_2 + f_0f_3 + 2f_2f_3 - f_3^2\\
 & -f_1f_2 + f_0f_3 + 2f_2f_3 - f_3^2 - 2f_3f_4\\
 &-f_1f_2 + f_0f_3 + 2f_2f_3 - f_3^2 - 2f_3f_4 + f_4^2\\
 & -f_1f_2 - f_2^2 + f_0f_3 + 2f_2f_3 - f_3^2 + f_0f_4 - 6f_3f_4 - f_4^2\\
&-f_2^2 + f_3^2 + f_0f_4 - f_2f_4 - 4f_3f_4 + 2f_4^2.\\
\end{align*}

\vskip .2in
Their $q$--expansions are given by the following expressions:

\begin{align*}
& q^{2}-2q^{8}-2q^{9}-2q^{10}+2q^{11}-4q^{13}+3q^{14}+4q^{15}+3q^{16}-2q^{17}+5q^{18} \\
& q^{3}-2q^{7}+q^{10}-2q^{11}+q^{12}-2q^{14}-4q^{16}+5q^{18} \\
& q^{4}-2q^{7}-q^{8}+3q^{9}+4q^{10}-4q^{11}-q^{13}-2q^{14}-3q^{15}-10q^{16}-2q^{17}+3q^{18} \\
& q^{5}-2q^{7}-q^{8}+3q^{9}+4q^{10}-4q^{11}-3q^{13}+q^{14}-q^{15}-11q^{16}-2q^{17}+5q^{18} \\
& q^{6}-2q^{11}-q^{12}-q^{13}-q^{14}+q^{15}-q^{16}+3q^{18} \\
& -2q^{7}+q^{8}+6q^{9}+q^{10}-10q^{11}-3q^{12}-5q^{13}+13q^{14}+21q^{15}-17q^{16}-8q^{17}-14q^{18} \\
& q^{8}+2q^{9}-5q^{10}-6q^{11}+19q^{12}+7q^{13}-13q^{14}-33q^{15}-7q^{16}+38q^{17}+14q^{18} \\
& 2q^{9}-q^{10}-4q^{11}+9q^{12}-5q^{13}+4q^{14}-13q^{15}-12q^{16}+18q^{17}+4q^{18} \\
& -q^{10}+11q^{12}-11q^{13}-7q^{15}-22q^{16}+22q^{17}+22q^{18} \\
& 11q^{12}-11q^{13}-11q^{15}-22q^{16}+22q^{17}+22q^{18} \\
&-22q^{13}+44q^{15}-44q^{16}+44q^{18} \\
&-22q^{14}+22q^{15}-22q^{16}+44q^{18} \\
\end{align*}

The last exponent is $14>  \frac{m}{2}+ (m-1)(g-1)-1=13$. So,  $\mathfrak a_\infty$ is not 
$2$--Weierstrass point for $X_0(55)$.

\section{ Wronskians of Modular Forms}\label{wron}

In this section we deal with a generalization of the usual notion of the Wronskian of cuspidal modular forms \cite{roh}, (\cite{ono}, 6.3.1),  (\cite{Muic}, the proof of Theorem 4-5),
and (\cite{Muic2}, Lemma 4-1). 

\vskip .2in

\begin{Lem}\label{wron-1} Let $f\in  M_m(\Gamma, \chi)$. Let $\gamma\in  \Gamma$.  Then, for $k\ge0$, $k$--th derivative of the function $f(\sigma.z)$ is given by
  $$
  \frac{d^{k}}{dz^k}f(\gamma.z) =\chi(\gamma)  j(\gamma, z)^{m+2k} \cdot \frac{d^{k}f(z)}{dz^{k}}+ \chi(\gamma) \sum_{i=0}^{k-1} D_{ik} \cdot j(\gamma, z)^{m+k+i} \cdot \frac{d^{i}f(z)}{dz^{i}}.
  $$
where $D_{ik}$ are some constants depending on $m$, $k$, and $\gamma$. If $\Gamma\subset SL_2(\mathbb Z)$, then the constants can be taken to be from $\mathbb Z$.
\end{Lem}
\begin{proof} This follows by an easy induction on $k$ using
  the fact that
 $$
\frac{d}{dz} \gamma.z=j(\gamma, z)^{-2}.
$$
See also the proof of Theorem 4-5 in\cite{Muic}, the text between the lines (4-6) and (4-8). 
\end{proof}

The following proposition is the main result of the present section:

\vskip .2in

\begin{Prop}\label{wron-2} Let $m\ge 1$. Then, for any sequence 
$f_1, \ldots, f_k\in M_m(\Gamma, \chi)$,
the Wronskian
$$
W\left(f_1, \ldots, f_k\right)(z)\overset{def}{=}\left|\begin{matrix}
f_1(z) &
\cdots & f_{k}(z) \\
\frac{df_1(z)}{dz} &
\cdots & \frac{df_{k}(z)}{dz} \\
&\cdots & \\
\frac{d^{k-1}f_1(z)}{dz^{k-1}} &
\cdots & \frac{d^{k-1}f_{k}(z)}{dz^{k-1}} \\
\end{matrix}\right|
$$
is a cuspidal modular form in $S_{k(m+k-1)}(\Gamma, \chi^k)$ if $k\ge 2$. If $f_1, \ldots, f_k$ are linearly independent,
then  $W\left(f_1, \ldots, f_k\right)\neq 0$.   
\end{Prop}
\begin{proof} This is a standard fact. We apply Lemma \ref{wron-1} to conclude
  $$
  W\left(f_1, \ldots, f_k\right)(\gamma. z)=\chi^k(\gamma) j(\gamma, z)^{k(m+k-1)} W\left(f_1, \ldots, f_k\right)(z), \ \ \gamma\in \Gamma, \ z\in \mathbb H.
  $$

 Let $x\in \bbR\cup \{\infty\}$ be a cusp for $\Gamma$. Let $\sigma\in SL_2(\bbR)$ such that
$\sigma.x=\infty$. We write 
$$
\{\pm 1\} \sigma \Gamma_{x}\sigma^{-1}= \{\pm 1\}\left\{\left(\begin{matrix}1 & lh'\\ 0 &
    1\end{matrix}\right); \ \ l\in \bbZ\right\},
$$
where $h'>0$ is the width of the cusp. Then we write the Fourier expansion of each $f_i$ at $x$ as follows:
$$
(f_i|_m \sigma^{-1})(\sigma.z)= \sum_{n=0}^\infty a_{n, i} \exp{\frac{2\pi
  \sqrt{-1}n\sigma.z}{h'}}.
$$

Using the cocycle identity
$$
1=j(\sigma^{-1}\sigma, z)=j(\sigma^{-1}, \sigma.z)j(\sigma, z),
$$
this implies the following:
$$
j(\sigma, z)^{m} \cdot f_i (z)=  \sum_{n=0}^\infty a_{n, i} \exp{\frac{2\pi \sqrt{-1}n\sigma.z}{h'}}.
$$

\vskip .1in

By induction on $t\ge 0$, using
$$
\frac{d}{dz} \sigma.z=j(\sigma, z)^{-2},
$$
we have the following:

\begin{equation}\label{wron-3}
j(\sigma, z)^{m+2t} \frac{d^{t}f_i(z)}{dz^{t}} + \sum_{u=0}^{t-1} D_{i, t} j(\sigma, z)^{m+t+u} \frac{d^{u}f_i(z)}{dz^{u}}=  \sum_{n=0}^\infty a_{n, i,t} \exp{\frac{2\pi
    \sqrt{-1}n\sigma.z}{h'}},
\end{equation}
for some complex numbers $D_{i, t}$ and $a_{n, i,t}$, where
$$
a_{0, i, t}=0, \ \ t\ge 1.
$$

\vskip .2in

Now, by above considerations, using (\ref{wron-3}),  we have 
\begin{align*}
  \left(W\left(f_1, \ldots, f_k\right)|_{k(m+k-1)} \sigma^{-1}\right)  (\sigma. z)
  &=j(\sigma, z)^{k(m+k-1)} W\left(f_1, \ldots, f_k\right)(z)\\
&= \left|\begin{matrix}
 j(\sigma, z)^{m} f_1(z) &
\cdots & j(\sigma, z)^{m} f_{k}(z) \\
j(\sigma, z)^{m+2} \frac{df_1(z)}{dz} &
\cdots & j(\sigma, z)^{m+2} \frac{df_{k}(z)}{dz} \\
&\cdots & \\
j(\sigma, z)^{m+2(k-1)}\frac{ d^{k-1}f_1(z)}{dz^{k-1}} &
\cdots & j(\sigma, z)^{m+2(k-1)} \frac{d^{k-1}f_{k}(z)}{dz^{k-1}} \\
  \end{matrix}\right|\\
  &= \det{\left( \sum_{n=0}^\infty a_{n, i+1,t} \exp{\frac{2\pi
        \sqrt{-1}n\sigma.z}{h'}}\right)_{0\le i, t\le k-1}}\\
\end{align*}

Now, we see that the Wronskian is holomorphic at each cup of $\Gamma$ and vanishes at the order at least $k-1$. In particular, it belongs to
$S_{k(m+k-1)}(\Gamma, \chi^k)$ if $k\ge 2$.

The claim that linear independence is equivalent to  the fact that Wronskian is not identically zero is standard (\cite{Miranda}, Chapter VII, Lemma 4.4).  
\end{proof}

\vskip .2in
We end this section with an  elementary remark regarding Wronskians. 
In the case when $\Gamma$ has a cusp at the infinity $\mathfrak a_\infty$, it is more convenient to use the derivative with respect to 
$$
q=\exp{\frac{2\pi  \sqrt{-1} z}{h}},
$$
where $h>0$ is the width of the cusp since all modular forms have $q$--expansions. Using the notation from Proposition \ref{wron-2}. It is easy to see
$$
\frac{d }{dz}= \frac{2\pi  \sqrt{-1} }{h} \cdot  q\frac{d }{dq}.
$$
This implies that

$$
\frac{d^k }{dz^k}= \left(\frac{2\pi  \sqrt{-1} }{h}\right)^k \cdot  \left(q\frac{d }{dq}\right)^k, \ \ k\ge 0.
$$

\vskip .2in 
Thus, we may define the $q$--Wronskian as follows:

\begin{equation}\label{wron-5000}
W_q\left(f_1, \ldots, f_k\right)\overset{def}{=}
 \left|\begin{matrix}
f_1 & \cdots & f_{k} \\
q\frac{d}{dq} f_1&
\cdots & q\frac{d }{dq} f_k\\
&\cdots & \\
\left(q\frac{d }{dq}\right)^{k-1} f_1&
\cdots & \left(q\frac{d }{dq}\right)^{k-1} f_k  \\
\end{matrix}\right|,
 \end{equation}
 considering $q$--expansions of $f_1, \ldots, f_k$.
 
\vskip .2in 

We obtain
\begin{equation}\label{wron-5001}
W\left(f_1, \ldots, f_k\right)=\left(\frac{2\pi  \sqrt{-1} }{h}\right)^{k(k-1)/2} W_q\left(f_1, \ldots, f_k\right).
 \end{equation}

\section{On a Divisor of a Wronskian} \label{wron-cont}

In this section we discuss the divisor of cuspidal modulars forms constructed via Wronskians (see Proposition \ref{wron-2}). We start with necessary
preliminary results.

\vskip .2in 

\begin{Lem} \label{wron-cont-1}
  Let $\varphi_1, \cdots, \varphi_l$ be a  sequence of linearly independent meromorphic functions on some open set $U\subset \mathbb C$. We define their Wronskian
  as usual $W(\varphi_1, \cdots, \varphi_k)= \det{\left(\frac{d^{i-1}\varphi_j}{dz^{i-1}} \right)_{i,j=1, \ldots, k}}$. Then, we have the following:
  \begin{itemize}
  \item[(i)]  The Wronskian $W(\varphi_1, \cdots, \varphi_k)$ is a non--zero meromorphic function on $U$.
  \item[(ii)] We have $W(\varphi_1, \cdots, \varphi_k)=\varphi^k W(\varphi_1/\varphi, \cdots, \varphi_{k}/\varphi)$ for all non--zero meromorphic functions $\varphi$ on $U$.
  \item[(iii)] Let $\xi\in U$ be such that all $\varphi_i$ are holomorphic. Let $A$  be the $\mathbb C$--span of all $\varphi_i$. Then, all $\varphi\in A$ are holomorphic at $\xi$,
    and the set $\{\nu_{z-\xi}(\varphi); \ \ \varphi\in A, \ \varphi\neq 0\}$ has exactly $k=\dim A$ different elements
    (Here as in Section \ref{prelim}, $\nu_{z-\xi}$ stands for the order at $\xi$.). Let
    $\nu_{z-\xi}(\varphi_1,\ldots, \varphi_k)$ be the sum of all $\dim A$--values of that set.   Then,  $W(\varphi_1, \cdots, \varphi_k)$ is holomorphic at $\xi$, and the
    corresponding order is
    $$
    \nu_{z-\xi}(\varphi_1,\ldots, \varphi_k)- \frac{k(k-1)}{2} 
    $$
    \end{itemize}
  \end{Lem}
\begin{proof} (i) is  well--known. See for example  (\cite{Miranda}, Chapter VII, Lemma 4.4) or it is a consequence of \cite[Proposition III.5.8]{FK}. (ii) is a consequence of the proof of
  \cite[Proposition III.5.8]{FK} (see formula (5.8.4)).   Finally, we prove (iii). Then,  by the text
  before the statement of \cite[Proposition III.5.8]{FK}, we see that we can select another basis $\psi_1, \ldots, \psi_k$ of $A$ such that
  $$
  \nu_{z-\xi}(\psi_1)< \nu_{z-\xi}(\psi_2)< \ldots< \nu_{z-\xi}(\psi_k).
  $$
  Then, by \cite[Proposition III.5.8]{FK}, we have that the order of $W(\psi_1, \ldots, \psi_k)$ at $z$ is equal to
  $$
    \sum_{i=1}^k \left( \nu_{z-\xi}(\psi_i) -i+1\right)= \nu_{z-\xi}(\varphi_1,\ldots, \varphi_k)- \frac{k(k-1)}{2}. 
    $$
    But    $\varphi_1, \ldots, \varphi_k$ is also a basis of $A$. Thus, we see that we can write
    $$
    \left(\begin{matrix}
      \varphi_1\\
      \varphi_2\\
      \vdots\\
      \varphi_k
      \end{matrix}\right)
    = A\cdot
    \left(\begin{matrix}
      \psi_1 \\
      \psi_2\\
      \vdots\\
      \psi_k
      \end{matrix}\right)
 $$   
 for some $A\in GL_k(\mathbb C)$. This implies
 $$
  \left(\begin{matrix}
      \varphi_1 &  d\varphi_1/dz& \cdots & d^{k-1}\varphi_1/dz^{k-1}\\
      \varphi_2 &  d\varphi_2/dz& \cdots & d^{k-1}\varphi_2/dz^{k-1}\\
      \vdots\\
      \varphi_k &  d\varphi_k/dz& \cdots & d^{k-1}\varphi_k/dz^{k-1}
      \end{matrix}\right)
    = A\cdot \left(\begin{matrix}
      \psi_1 &  d\psi_1/dz& \cdots & d^{k-1}\psi_1/dz^{k-1}\\
      \psi_2  & d\psi_2/dz& \cdots & d^{k-1}\psi_2/dz^{k-1}\\
      \vdots\\
      \psi_k  & d\psi_k/dz& \cdots & d^{k-1}\psi_k/dz^{k-1}
      \end{matrix}\right).
    $$
    Hence
    $$
    W(\varphi_1, \ldots, \varphi_k) =\det A \cdot W(\psi_1, \ldots, \psi_k)
    $$
    has the same order at $z$ as $W(\psi_1, \ldots, \psi_k)$.
  \end{proof}

\vskip .2in
As a direct consequence of Lemma \ref{wron-cont-1}, we obtain the following result. At this point the reader should review the text in Section \ref{prelim} before the statement of Lemma \ref{prelim-1}
as well as Proposition \ref{wron-2}. The proof is left to the reader as an exercise. 

\begin{Prop}\label{wron-cont-2}
  Assume that $m\ge 2$ is even. Let  $f_1, \ldots, f_k\in M_m(\Gamma)$ be a sequence of linearly independent modular forms. Let $\xi \in \mathbb H$. Then, we have the following:
  $$
  \nu_{\mathfrak a_\xi}\left(W(f_1,\ldots, f_k)\right)=\frac{1}{e_\xi} \cdot \left( \nu_{z-\xi}(\varphi_1,\ldots, \varphi_k)- \frac{k(k-1)}{2}\right).
  $$
  \end{Prop}

\vskip .2in
The case of a cusp requires a different technique but final result is similar:

\begin{Thm}\label{wron-6} Assume that $m\ge 2$ is even. Suppose that $\mathfrak a_\infty$ is a cusp for $\Gamma$.
  Let  $f_1, \ldots, f_k\in M_m(\Gamma)$ be a sequence of linearly independent modular forms. Let $i\in \{1, \ldots, k\}$.   Consider
  $f_1, \ldots, f_k$ as meromorphic functions in a variable $q$ in a neighborhood of $q=0$, and define $\nu_{q-0}\left(f_1, \ldots, f_k\right)$ as 
 in Lemma \ref{wron-cont-1} (iii). Then, we have the following identity:
  $$
  \nu_{\mathfrak a_\infty}\left(W\left(f_1, \ldots, f_k\right)\right)=    \nu_{q-0}\left(f_1, \ldots, f_k\right).
  $$
 \end{Thm} 
\begin{proof} By Lemma \ref{wron-cont-1} (ii),  we can write 
\begin{equation}\label{wron-5} 
W\left(f_1, \ldots, f_k\right) =f_1^k \cdot W(1, f_2/f_1, \ldots, f_k/f_1)
\end{equation}
as meromorphic functions on $\mathbb H$.  But the key fact that $1, f_2/f_1, \ldots, f_k/f_1$  
can be regarded as meromorphic (rational) functions on $\mathfrak R_\Gamma$ i.e., they are elements of 
$\mathbb C\left(\mathfrak R_\Gamma\right)$.

The key point now is that these meromorphic functions, and their Wronskian define meromorphic  $k(k-1)/2$--differential form,
denoted by $W_\Gamma$. Details are contained in \cite[Section 4, Lemma 4.9]{Miranda}.  We recall the following.

Let $w\in \mathbb H$ be a non--elliptic point for $\Gamma$,
such that $f_1\neq 0$,  and $U\subset \mathbb H$ small neighborhood of $w$ giving a chart of $\mathfrak a_w$ on the curve $\mathfrak R_\Gamma$. Then, in the chart $U$ we have:
$$
W_\Gamma= W(1, f_2(z)/f_1(z), \ldots, f_k(z)/f_1(z)) \left(dz\right)^{k(k-1)/2}.
$$
On the other hand, in a chart of $\mathfrak a_\infty$, $W_\Gamma$  is given by
the usual Wronskian, denoted by $W_{\Gamma, q}(1, f_2/f_1, \ldots, f_k/f_1)$, 
of $1, f_2/f_1, \ldots, f_k/f_1$ presented by $q$--expansions with respect to the derivatives $d^i/dq^i$,  $0\le i\le k-1$, multiplied by 
$\left(dq\right)^{k(k-1)/2}$ i.e.,
$$
W_\Gamma= W(1, f_2/f_1, \ldots, f_k/f_1) \left(dq\right)^{k(k-1)/2}.
$$

 Next, we insert $q$-expansions of $f_1, \ldots, f_k$ into $W(1, f_2/f_1, \ldots, f_k/f_1)$. So, we can express
  $$
  W(1, f_2/f_1, \ldots, f_k/f_1)=c_mq^m+ c_{m+1} q^{m+1}+\cdots,
  $$
  where $c_m\neq 0$, $c_{m+1}, c_{m+2}, \ldots$ are complex numbers. Hence,
  \begin{equation}\label{wron-6000}
   \nu_{\mathfrak a_\infty}\left(W(1, f_2/f_1, \ldots, f_k/f_1)\right) = \nu_{q-0}\left( W(1, f_2/f_1, \ldots, f_k/f_1)\right)=m.
    \end{equation}

  Let us fix a neighborhood $U$ of $\infty$ such that it is  a chart for  $\mathfrak a_\infty$, and there is no elliptic points in it. Then, we fix
  $w\in U$, $w\ne \infty$, and a chart $V$ of $w$ such that $V\subset \mathbb H\cap U$. Now,   on $V$, we have   the following expression for $W_\Gamma$:
  $$
   W(1, f_2/f_1, \ldots, f_k/f_1) \left(dz\right)^{k(k-1)/2}=\left(c_mq^m+ c_{m+1} q^{m+1}+\cdots\right) \left(dz\right)^{k(k-1)/2}, \ \ q=\exp{\frac{2\pi  \sqrt{-1} z}{h}}.
  $$
  On the other hand, on $U$, we must have the expression for $W_\Gamma$ of the form
  $$
  W_{\Gamma, q}(1, f_2/f_1, \ldots, f_k/f_1) \left(dq\right)^{k(k-1)/2}=  \left(d_nq^n+ d_{n+1} q^{m+1}+\cdots\right) \left(dq\right)^{k(k-1)/2},
  $$
  where $d_n\neq 0$, $d_{n+1}, d_{n+2}, \ldots$ are complex numbers. We have 
  \begin{equation}\label{wron-6001}
     \nu_{\mathfrak a_\infty}\left(W_\Gamma\right)=   \nu_{q-0}\left( W_{\Gamma, q}(1, f_2/f_1, \ldots, f_k/f_1)\right)=n.
    \end{equation}

  By definition of meromorphic $k(k-1)/2$--differential, on $V$ these expressions must be related by
  $$
  c_mq^m+ c_{m+1} q^{m+1}+\cdots =  \left(d_nq^n+ d_{n+1} q^{m+1}+\cdots\right) \cdot \left(\frac{dq}{dz}\right)^{k(k-1)/2} .
  $$
  Hence, we obtain
  $$
  n=m-k(k-1)/2.
  $$
  Using (\ref{wron-6000}) and (\ref{wron-6001}) this can be written as follows: 
  $$
  \nu_{q-0}\left( W_{\Gamma, q}(1, f_2/f_1, \ldots, f_k/f_1)\right)=  \nu_{q-0}\left( W(1, f_2/f_1, \ldots, f_k/f_1)\right) -k(k-1)/2.
  $$

  Consider again    $f_1, \ldots, f_k$ as meromorphic functions in a variable $q$ in a neighborhood of $q=0$, and define
  the Wronskian $W_{\Gamma, q}(f_1, \ldots, f_k)$ using derivatives with respect to $q$. Then, Lemma \ref{wron-cont-1} (ii) implies
  \begin{align*}
    \nu_{\mathfrak a_\infty}\left(W\left(f_1, \ldots, f_k\right)\right)&= \nu_{\mathfrak a_\infty}\left(f^k_1\right)+
    \nu_{\mathfrak a_\infty}\left(W\left(1, f_2/f_1, \ldots, f_k/f_1\right)\right)\\
    &= \nu_{q-0}\left(f^k_1\right)+
    \nu_{q-0}\left(W\left(1, f_2/f_1, \ldots, f_k/f_1\right)\right)\\
    &= \nu_{q-0}\left(f^k_1\right)+
    \nu_{q-0}\left( W_{\Gamma, q}(1, f_2/f_1, \ldots, f_k/f_1)\right)+k(k-1)/2\\
    &= \nu_{q-0}\left( W_{\Gamma, q}(f_1, f_2, \ldots, f_k)\right)+k(k-1)/2.
    \end{align*}
  Finally, we apply Lemma \ref{wron-cont-1} (iii).  
   \end{proof}

\section{Computation of Wronskians for $\Gamma=SL_2(\mathbb Z)$}\label{lev0}

Assume that $m\ge 4$ is an even integer. 
Let  $M_m$  be the space of  all modular forms of weight $m$
for $SL_2(\mathbb Z)$. We introduce the two Eisenstein series
\begin{align*}
&E_4(z)=1+240 \sum_{n=1}^\infty \sigma_3(n)q^n\\
&E_6(z)=1 -504 \sum_{n=1}^\infty \sigma_5(n)q^n\\
\end{align*}
of weight $4$ and $6$, where $q=\exp{(2\pi i z)}$. Then, for any even integer $m\ge 4$, we have
\begin{equation}\label{lev0-1}
M_m=\oplus_{\substack{\alpha, \beta\ge 0\\ 4\alpha+6\beta =m }} \mathbb C E^\alpha_4 E^\beta_6.
\end{equation}
We have

\begin{equation}\label{lev0-2}
k=k_m\overset{def}{=} \dim M_m=\begin{cases} \left[m/12\right]+1, \ \ m\not\equiv \ 2 \text{(mod $12$)};\\
\left[m/12\right], \ \ m\equiv \ 2 \text{(mod $12$)}.\\
\end{cases}
\end{equation}

\vskip .2in
We let 
$$
\Delta(z)=q+\sum_{n=2}\tau(n)q^n=q -24q^2+252 q^3+\cdots =\frac{E_4^3(z)-E_6^2(z)}{1728}. 
$$
be the Ramanujan delta function.

\vskip .2in
It is well--known that the map $f\longmapsto f\cdot \Delta$ is an ismorphism between the vector space of modular form $M_m$ and the space of all cuspidal modular forms $S_{m+12}$ inside
$M_{m+12}$. In general, we have the following:

$$
\dim S_m=\dim M_m-1,\\
$$
for all even integers $m\ge 4$.

\vskip .2in
Now, we are ready to compute our first Wronskian (see (\ref{wron-5000}) for notation).
\vskip .2in

\begin{Prop}\label{lev0-3}  We have the following:
  \begin{itemize}
  \item[(i)] $W_q\left(E^3_4, E^2_6\right)=-1728\cdot \Delta \cdot  E^2_4 E_6$.  
    \item[(ii)]  $2 E_4  \frac{d}{dq}E_6  -  3  E_6 \frac{d}{dq} E_4=-1728 \cdot \Delta\cdot q^{-1}$.
  \end{itemize}
  \end{Prop}
\begin{proof} We compute
  \begin{align*}
    W_q\left(E_4, E_6\right) &=
\left|\begin{matrix} E^3_4 & E^2_6\\
q\frac{d}{dq} E^3_4   & q\frac{d}{dq}  E^2_6
\end{matrix}\right|\\
&=2 E^3_4 E_6 \cdot q\frac{d}{dq}E_6  -  3  E^2_4  E^2_6 \cdot q\frac{d}{dq} E_4\\
&= E^2_4 E_6 \cdot q\cdot \left(2 E_4  \frac{d}{dq}E_6  -  3  E_6 \frac{d}{dq} E_4 \right).
  \end{align*}
  
But we know that $W_q\left(E_4, E_6\right)$ is a cusp form of weigth $2\cdot (12+2-1)=26$. Thus, we must have that is equal to
$$
W_q\left(E_4, E_6\right) =\lambda\cdot \Delta\cdot E^2_4 E_6,
$$
for some non--zero constant $\lambda$.  This implies that

$$
2 E_4  \frac{d}{dq}E_6  -  3  E_6 \frac{d}{dq} E_4 =\lambda \cdot \Delta \cdot q^{-1}
$$
Considering explicit $q$--expansions, we find that
$$
\lambda=-1728.
$$
This proves both (i) and (ii). 
\end{proof}

\vskip .2in
The general case requires a different proof based on results of Section \ref{wron-cont}.

  \vskip .2in
  \begin{Prop} \label{lev0-4}  Assume that $m=12t$ for some $t\ge 1$. Then, we write the basis of $M_m$ as follows: $\left(E^3_4\right)^u \left(E^2_6\right)^{t-u}$, $0\le u \le t$.
    Then, we have the following 
    $$
    W_q\left(\left(E^3_4\right)^u \left(E^2_6\right)^{t-u}, \ \ 0\le u \le t\right)= \lambda \cdot\Delta^{\frac{t(t+1)}{2}}  E_4^{t(t+1)}E_6^{\frac{t(t+1)}{2}},
    $$
    for some non--zero constant $\lambda$. 
  \end{Prop}
  \begin{proof} We can select another basis $f_0, \ldots, f_{t}$ of $M_m$ such that $f_i=c_i q^i+ d_i q^{i+1}+\cdots$, $0\le i \le t$, where $c_i\neq 0 , d_i, \ldots$ are some complex constants.
   An easy  application of  Theorem \ref{wron-6} gives 
$$
\nu_{\mathfrak a_\infty}\left(W_{q}\left(\left(E^3_4\right)^u \left(E^2_6\right)^{t-u}, \ \ 0\le u \le t\right)\right)= \frac{t(t+1)}{2}.
$$
But since $\mathrm{div}{(\Delta)}=\mathfrak a_\infty$, we obtain that
  $$
  f\overset{def}{=}W_{q}\left(\left(E^3_4\right)^u \left(E^2_6\right)^{t-u}, \ \ 0\le u \le t\right)/\Delta^{\frac{t(t+1)}{2}}
  $$
is a non--cuspidal modular form of weight 
  $$
  l=k\cdot \left(m+ k-1\right) -12\frac{t(t+1)}{2}=(t+1) (12t+t)-12t=7t(t+1).
  $$

  It remains to determine $f$. In order to do that,  we use  Proposition \ref{wron-cont-2}, and consider the order of vanishing of
  $W_q\left(\left(E^3_4\right)^u \left(E^2_6\right)^{t-u}, \ \ 0\le u \le t\right)$ at elliptic points $i$ and $e^{\pi i/3}=(1+i\sqrt{3})/2$, of order $2$ and $3$, respectively.
  We recall (see \cite{Muic2}, Lemma 4-1) that
  $$
  \mathrm{div}{(E_4)}=\frac13 \mathfrak a_{(1+i\sqrt{3})/2}.
  $$
  Similarly we show that
  $$
  \mathrm{div}{(E_6)}=\frac12 \mathfrak a_{i}.
  $$
  This implies that $\left(E^3_4\right)^u \left(E^2_6\right)^{t-u}$ has order $3u$ and $2(t-u)$ at $(1+i\sqrt{3})/2$ and $i$, respectively.
  Hence,  $W_q\left(\left(E^3_4\right)^u \left(E^2_6\right)^{t-u}, \ \ 0\le u \le t\right)$ has orders
  $$
  \nu_{\mathfrak a_{(1+i\sqrt{3})/2}}\left(W_q\left(\left(E^3_4\right)^u \left(E^2_6\right)^{t-u}, \ \ 0\le u \le t\right)\right)=\frac{1}{3} t(t+1),
  $$
  and 
  $$
  \nu_{\mathfrak a_{(1+i\sqrt{3})/2}}\left(W_q\left(\left(E^3_4\right)^u \left(E^2_6\right)^{t-u}, \ \ 0\le u \le t\right)\right)=\frac{1}{4} t(t+1).
  $$
  This implies the following: 
 $$
  \nu_{\mathfrak a_{(1+i\sqrt{3})/2}}\left(f\right)  =\frac{1}{3} \cdot t(t+1),
  $$
  and
   $$
  \nu_{\mathfrak a_i}\left(f\right)  =\frac{1}{4} \cdot t(t+1),
  $$
  Since, $f\in M_{7t(t+1)}$,   comparing divisors as before, we conclude that
  $$
  f=\lambda \cdot E_4^{t(t+1)}E_6^{\frac{t(t+1)}{2}},
  $$
  for some non--zero constant $\lambda$. 
  \end{proof}

  \vskip .2in
  We are not able to determine constant $\lambda$ in Proposition \ref{lev0-4} for all $t\ge 1$. It should come out of comparison of $q$--expansions
  of left and right sides of the identity in Proposition \ref{lev0-4}. For $t=1$, Proposition \ref{lev0-3} implies that $\lambda=-1728$. 
  Experiments in SAGE shows that $\lambda= -2 \cdot 1728^3$ for $t=2$, and  $\lambda= 12 \cdot 1728^6$ for $t=3$.


\begin{thebibliography}{999999}

    \bibitem{BKS} {\sc R.~Br\" oker, K.~Lauter, A.~V.~Sutherland,}{\em    Modular polynomials via isogeny volcanoes,} 
	Mathematics of Computation {\bf 81}, 1201--1231 ,  (2012)
	
	\bibitem{bnmjk} {\sc B.~Cho, N.~M.~Kim, J.~K.~Koo,}
	{\em Affine models of the modular curves $X(p)$ and its application,} 
	Ramanujan J. {\bf 24} , no. 2, 235–-257, (2011) 
	
	
	
	
	\bibitem{sgal}
	{\sc S.~Galbraith,} {\em  Equations for modular curves, Ph.D. thesis,} Oxford (1996)
	
	\bibitem{ishida}
	{\sc N.~Ishida,} {\em Generators and equations for modular function fields of principal congruence subgroups,} Acta Arithmetica, {\bf 85}  no 3, 197--207, (1998)
	

  \bibitem{FK}{\sc H.~M.~Farkas, I.~Kra,} {\em Riemann surfaces. Second edition. Graduate Texts in Mathematics,} {\bf 71.} Springer-Verlag, New York, 1992.
    

  	\bibitem{Kodrnja1}
	{\sc I.~ Kodrnja, }{\em On a simple model of $X_0(N)$},  Monatsh. Math. 186 (2018), no. 4, 653--661.

        
      \bibitem{Miyake}
        {\sc T.~Miyake,} {\em Modular forms,} Springer-Verlag, (2006)

  	\bibitem{ArSe}  {\sc  E.~Arbarello, E.~Sernesi,} {\sc Petri's approach to the study of the ideal associated to a special divisor,} Invent. Math., {\bf 49} (1978) pp. 99--119.
          
        
 \bibitem{Miranda} {\sc R.~Miranda,}{\em Algebraic Curves and Riemann Surfaces,} 
	Graduate Studies in Mathematics {\bf 5}, (1995)
	
	\bibitem{Muic} {\sc G.~Mui\' c}, 
	{\em Modular curves and bases for the spaces of cuspidal modular forms,}
	Ramanujan J. {\bf 27}, 181–-208,  (2012)
	
	
	\bibitem{Muic1} {\sc G.~Mui\' c,} {\em On embeddings of curves in projective spaces,} 
	Monatsh. Math. {\bf  Vol. 173, No. 2}, 239--256,  (2014)
	
	
	\bibitem{Muic2} {\sc G.~Mui\' c,} {\em On degrees and birationality of the maps $X_0(N)\rightarrow \mathbb P^2$ constructed via modular forms, } 
	  Monatsh. Math. {\bf  Vol. 180, No. 3}, 607--629,  (2016)

        \bibitem{MuicKodrnja} {\sc G.~Mui\' c, I.~Kodrnja} {\em On primitive elements of algebraic function fields and models of $X_0(N)$}, preprint (https://arxiv.org/abs/1805.02112)

        \bibitem{MuicMi} {\sc G.~Mui\' c, D.~Miko\v c,} {\em Birational maps of $X(1)$ into $\mathbb P^2$,} Glasnik Matematicki {\bf Vol. 48}, No. 2, 301--312,  (2013).


          \bibitem{mshi} {\sc M.~Shimura,}{\em Defining Equations of Modular Curves $X_{0}(N)$,} 
	Tokyo J. Math. {\bf Vol. 18,} No. 2, (1995)

        
          
       
\bibitem{neeman} {\sc A.~Neeman,}{\em  The distribution of Weierstrass points on a compact Riemann surface,}
Ann. Math. {\bf 120} (1984), 317--328.

\bibitem{Ogg} {\sc A.~P.~Ogg,} {\em Hyperelliptic modular curves.} Bull. Soc. Math. France 102 (1974), 449–462.
\bibitem{olsen} {B.~A.~Olsen,} {\em On Higher Order Weierstrass Points,} Ann.
Math.{\bf 95}  No. 2 (1972), 357--364.

\bibitem{ono} {\sc K.~Ono,} {\em The Web of Modularity: Arithmetic of the Coefficients of 
		Modular Forms and $q$--Series,} Conference Board of the Mathematical Sciences {\bf 102}, 
	American Mathematical Society (2004).


\bibitem{pete-1} {\sc H.~Petersson,} { \em Einheitliche Begr\" undung der Vollst\" andigkeitss\" atze 
f\" ur die Poincar\'eschen Reihen von 
reeller Dimension bei beliebigen Grenzkreisgruppen von erster Art, } Abh. Math. Sem. Hansischen 
Univ. {\bf 14} (1941), 22--60. 



\bibitem{pete-2} {\sc H.~Petersson,} { \em \" Uber Weierstrasspunkte und die expliziten Darstellungen der 
automorphen Formen von reeller Dimension,} Math. Z. {\bf 52} (1949), 32–-59.


\bibitem{roh} {\sc D.~Rohrlich,} {\em  
Weierstrass points and modular forms,} Illinois J. Math  
{\bf 29} (1985),  134-141


          
	\bibitem{SAGE} {\em Sage Mathematics Software (Version 8.8), The Sage Developers, 2019,
	  http://www.sagemath.org.}

          \bibitem{yy} {\sc Y.~Yifan,} {\em Defining equations of modular curves,} Advances in Mathematics {\bf 204}, 481-–508,  (2006)

\end{thebibliography}
\end{document}